\documentclass[11pt,twoside]{preprint}

\usepackage{hyperref}
\usepackage{breakurl}

\usepackage{times}
\usepackage{microtype}
\usepackage{amsmath}
\usepackage{amssymb}
\usepackage{tikz}
\usepackage{wasysym}
\usepackage{centernot}
\usepackage{color}

\usepackage{mathrsfs}
\usepackage{mhequ}
\usepackage{mhenvs}
\usepackage{mhsymb}
\usepackage[top=4cm, bottom=3.5cm, left=3.5cm, right=3.2cm]{geometry}

\colorlet{symbols}{blue!90!black}
\def\symbol#1{\textcolor{symbols}{#1}}
\def\1{\mathbf{\symbol{1}}}

\def\emptyset{\mathop{\centernot\ocircle}}

\colorlet{testcolor}{green!60!black}

\newtheorem{example}[lemma]{Example}

\definecolor{darkred}{rgb}{0.9,0.1,0.1}

\def\comment#1{\ifthenelse{\isodd{\value{page}}}{\marginpar{\raggedright\scriptsize{\textcolor{darkred}{#1}}}}{\marginpar{\raggedleft\scriptsize{\textcolor{darkred}{#1}}}}}  

\newop{diag}

\let\D\CD

\def\cE{\mathcal{E}}
\def\cF{\mathcal{F}}

\def\EE{{\mathscr E}}
\def\FF{{\mathscr F}}

\def\ss{{\mathtt{s}}}
\def\s{{\mathrm{s}}}
\def\e{{\mathrm{e}}}

\def\i{{\mathrm{i}}}
\def\L{{\mathcal{L}}}

\def\limsup{\mathop{\overline{\mathrm{lim}}}}
\def\liminf{\mathop{\underline{\mathrm{lim}}}}

\def\${|\!|\!|}
\def\l|{\left|\!\left|\!\left|}
\def\r|{\right|\!\right|\!\right|}

\begin{document}

\title{On stiff problems via Dirichlet forms}
\author{Liping Li$^1$, Wenjie Sun$^2$}
\institute{RCSDS, HCMS, Academy of Mathematics and Systems Science, Chinese Academy of Sciences, \email{liliping@amss.ac.cn} \and Fudan University, \email{wjsun14@fudan.edu.cn}}

\maketitle

\begin{abstract}
The stiff problem is concerned with a thermal conduction model with a singular barrier of zero volume. In this paper, we shall build the phase transitions for the stiff problems in one-dimensional space. It turns out that every phase transition definitely depends on the total thermal resistance of the barrier, and the three phases correspond to the so-called impermeable pattern, semi-permeable pattern and permeable pattern of thermal conduction respectively. For each pattern, the related boundary condition of the flux at the barrier is also derived. Mathematically, we shall introduce and explore the so-called snapping out Markov process, which is the probabilistic counterpart of semi-permeable pattern in the stiff problem. 
\end{abstract}

\tableofcontents

\section{Introduction}

The stiff problem (Cf. \cite{SanchezPalencia:1980dg}) is concerned with a thermal conduction model with a singular barrier. In \cite{L16}, the terminology `thin layer problem' was used instead.  Let us explain it by a concrete example in one-dimensional space. Given a small constant $\varepsilon >0$, consider the following heat equation:
\begin{equation}\label{EQ3PTU}
	\partial_t u^\varepsilon(t,x)=\frac{1}{2}\nabla \left(a_\varepsilon(x)\nabla u^\varepsilon(t,x) \right),\quad t\geq 0, x\in \mathbb{R}
\end{equation}
with the initial condition $u^\varepsilon(0,\cdot)=u_0$. Note that $a_\varepsilon$ is the so-called  thermal conductivity (or diffusive coefficient). A small normal barrier $I_\varepsilon$ is put near $0$ in the sense that $a_\varepsilon$ is very small in $I_\varepsilon$. In \cite{L16}, it is further assumed that $a_\varepsilon$ is constant either in or out of  $I_\varepsilon$, and the small thermal conductivity in $I_\varepsilon$ has the same scale as the length of  $I_\varepsilon$. More precisely, 
\begin{equation}\label{EQ1AXA}
I_\varepsilon:=(-\varepsilon, \varepsilon),\quad a_\varepsilon(x):=\left\{\begin{aligned} 
1,\quad x\notin (-\varepsilon, \varepsilon),\\
\kappa \varepsilon,\quad x\in (-\varepsilon,\varepsilon),
\end{aligned} \right.
\end{equation}
for a fixed constant $\kappa>0$ as in \cite{L16}. 
Then the limit of $u^\varepsilon$ is expected in the stiff problem as $\varepsilon\downarrow 0$. 
Heuristically speaking, the singular barrier (at $0$) is thought of as a material of zero length and zero thermal conductivity in this thermal conduction. One can prove that $u^\varepsilon$ converges to a function $u$ satisfying
\begin{equation}\label{EQ3TVT}
	\partial_tu(t,x)=\frac{1}{2}\Delta u(t,x),\quad u(0,x)=u_0(x)
\end{equation}
and the discontinuity of the flux at $0$: 
\begin{equation}\label{EQ3VTV}
\nabla u(t,0+)=\nabla u(t,0-)=\frac{\kappa}{2}(u(t,0+)-u(t,0-))
\end{equation}
in a certain meaning ($u$ is also called the \emph{flux}).

On the other hand, to our knowledge, it was  Lejay, who first studied the probabilistic description of this stiff problem in \cite{L16}. For any fixed $\varepsilon>0$, it is well known that \eqref{EQ3PTU} with $a_\varepsilon$ in \eqref{EQ1AXA} has an associated diffusion process $(X^\varepsilon_t)_{t\geq 0}$ on $\mathbb{R}$ such that 
\[
	u^\varepsilon(t,x)=\mathbf{E}_x u_0(X^\varepsilon_t).
\]
Needless to say, it is surely interesting to ask whether $X^\varepsilon$ could converge to some process as $\varepsilon\downarrow 0$, and if the limit exists, how it links the heat equation \eqref{EQ3TVT} and the boundary condition \eqref{EQ3VTV}. As we have known, the snapping out Brownian motion (SNOB in abbreviation) introduced in \cite{L16} is the desired limit. It is a Feller process on $\mathbb{G}:=(-\infty, 0-]\cup [0+,\infty)$, in which $0\in \mathbb{R}$ corresponds to two distinct points. Roughly speaking, the SNOB denoted by $(Y_t)_{t\geq 0}$ behaves like a reflecting Brownian motion on $\mathbb{G}_-:=(-\infty, 0-]$ or $\mathbb{G}_+:=[0+,\infty)$ and may change its sign and start as a new reflecting Brownian motion on the other component of $\mathbb{G}$ by chance, when it hits $0+$ or $0-$.  Lejay believed that $X^\varepsilon$ converges to the SNOB and 
\[
	u(t,x)=\mathbf{E}_x u_0(Y_t)
\]
satisfies \eqref{EQ3TVT} and \eqref{EQ3VTV} in some sense. In practice, he proved that the resolvent of SNOB satisfies the boundary condition \eqref{EQ3VTV} and another process $Z^\varepsilon$, a censored version of $X^\varepsilon$ obtained by a special transform, converges to the SNOB as $\varepsilon\downarrow 0$. 

The main purpose of this paper is to explore the general stiff problems and their probabilistic counterparts by means of Dirichlet forms. Let us introduce the background of Dirichlet forms upfront. A Dirichlet form is a symmetric Markovian closed form on $L^2(E,m)$ space, where $E$ is a nice topological space and $m$ is a Radon measure on it. Theory of Dirichlet form is closely related to the probability theory because of its Markovian property. Due to a series of important works by Fukushima, Silverstein in 1970's and Albeverio, Ma and R\"ockner in 1990's, it is now well known that a regular (resp. quasi-regular) Dirichlet form is always associated with a symmetric Markov process. We refer the notions and terminologies in theory of Dirichlet form to \cite{CF12, FOT11}. 

As mentioned above,  Lejay only considered the Brownian case of stiff problems, in which the conductivity is constant out of the barrier. 
His approach to the SNOB is based on the resolvent analysis of elastic Brownian motion, which is a perturbation of two-sided reflecting Brownian motion on $\mathbb{G}$, and the SNOB is eventually obtained by applying the piecing out transform (Cf. \cite{INW66}) to the elastic Brownian motion. Though the idea is heuristic, this approach is a little cumbersome and hard to generalize. Approach of Dirichlet form proposed by us is another possible way to obtain the SNOB. As we know, Dirichlet form is a very powerful tool to deal with the general Markov process and its related  probabilistic notions. For example, the perturbation in elastic Brownian motion is a special case of so-called killing transform for a general Markov process, and in theory of Dirichlet form, the killing transform is described by the perturbed Dirichlet form illustrated in \S\ref{SEC21}. Moreover, by an argument of resolvent analysis on $L^2(E,m)$, we can also derive the Dirichlet form of piecing out method in Theorem~\ref{THM35}. Particularly, the SNOB is associated with a regular Dirichlet form on $L^2(\mathbb{G},m)$ as follows:
\[
\begin{aligned}
 & \FF^\mathrm{s}=\left\{u\in L^2(\mathbb{G},m): 
 u_+\in H^1\left([0+,\infty)\right), ~u_-\in H^1\left((-\infty, 0-]\right)\right\} \\
 &\EE^\mathrm{s}(u,v)=
 \frac{1}{2}\int_\mathbb{G} u'(x)v'(x)dx+\dfrac{\kappa}{4}(u(0+)-u(0-))(v(0+)-v(0-)),\quad u,v\in \FF^\mathrm{s},
 \end{aligned}
\]
where $m$ is the Lebesgue measure on $\mathbb{G}$ and $u_+:=u|_{[0+,\infty)}, u_-:=u|_{(-\infty, 0-]}$. This indicates that the switches of SNOB at $0$ are essentially the additional jumps between $0+$ and $0-$. After the generator of SNOB on $L^2(E,m)$ is put forward in Proposition~\ref{PRO45}, the relation between SNOB and \eqref{EQ3VTV} also becomes clear, since $u_t(x)=\mathbf{E}_x u_0(Y_t)$ is a continuous function (on $\mathbb{G}$) belonging to $\FF^\s$ for $t>0$. The arguments based on Dirichlet forms are valid not only for the Brownian case, but also for a rich class of thermal conduction models. In practice, we shall characterize the associated Markov process and related boundary condition of the flux at $0$ for the stiff problem with a lower and upper bounded conductivity in \S\ref{SEC44}. 

The extension of SNOB is a reason to start this paper, but it is not the most important reason. In the Brownian case, the form of the conductivity in \eqref{EQ1AXA} is a little incomprehensible from Lejay's approach. Primarily, it is not easy to find a sensible physical interpretation of the assumption that $a_\varepsilon$ has the same scale as $\varepsilon$ in $I_\varepsilon$. Approach of Dirichlet form could shed light on the essence of this assumption, and this is the principal reason that initiates this article. To show this, let us use a few lines to summerize the characterization of one-dimenisonal diffusions. It is well known that under a `regularity' condition, a diffusion on $\mathbb{R}$  with no killing inside could be characterized essentially by a function $\ss$, called scale function and a measure $m$, called speed measure (Cf. \cite{IM74}). In this case the speed measure is also the unique symmetric measure. Note that the scale function is a continuous and strictly increasing function and induces a fully supported positive Radon measure $\lambda$ on $\mathbb{R}$.  It is performed in \cite{FHY10, F14, LY172} that the Dirichlet form (on $L^2(\mathbb{R},m)$) of this diffusion is completely characterized by $\lambda$ (as well as $\ss$) as follows (see \eqref{EQ4FIF})
\begin{equation}\label{EQ1FFL}
\begin{aligned}
&\FF^\i=\bigg\{f\in L^2(\mathbb{R}, m): f\ll \lambda, \int_{\mathbb{R}} \left(\frac{df}{d\lambda}\right)^2d\lambda<\infty, \\
&\qquad\qquad\qquad\qquad \qquad f(\pm\infty):=\lim_{x\rightarrow \pm\infty}f(x)=0 \text{ if }\lambda_\pm(\mathbb{G}_\pm)<\infty\bigg\}, \\
&\EE^\i(f,g)=\frac{1}{2}\int_{\mathbb{R}}\frac{df}{d\lambda}\frac{dg}{d\lambda}d\lambda,\quad f,g\in \FF^\i. 
\end{aligned}
\end{equation}
As Dirichlet form stands for the energy of associated generator, $\lambda$ plays the role of the `thermal resistance', which reflects the ability of the material to resist the flow of the heat (see Remark~\ref{RM43}). Thereupon, the general stiff problem in one-dimensional space can be reintroduced in the manner of thermal resistance as follows. Recall $I_\varepsilon=(-\varepsilon,\varepsilon)$ and declare $\gamma_\varepsilon$ to be a finite Radon measure on $I_\varepsilon$ with full support charging no set of singleton, i.e. $\gamma_\varepsilon(\{x\})=0$ for any $x\in I_\varepsilon$. Another measure $\lambda_\varepsilon$ is, by definition, equal to $\gamma_\varepsilon$ on $I_\varepsilon$ and equal to $\lambda$ outside $I_\varepsilon$ (see \eqref{EQ4LVT}). The diffusion $X^\varepsilon$ with scale function induced by $\lambda_\varepsilon$ corresponds to a thermal conduction model with the small barrier $(I_\varepsilon, \gamma_\varepsilon)$. Then the stiff problem is concerned with the convergence of $X^\varepsilon$ as well as the related flux as $\varepsilon\downarrow 0$. The following heuristic observation gives insight to this stiff problem:
\[
	\lambda_\varepsilon \rightarrow \lambda+\bar{\gamma}\cdot \delta_0,\quad \text{as } \varepsilon\downarrow 0,
\]
where $\bar{\gamma}:=\lim_{\varepsilon\downarrow 0}\gamma_\varepsilon(I_\varepsilon)$ is called the total thermal resistance of the singular barrier (Figure~\ref{FIG1} is an illustration of this observation, in which $\lambda_\pm:=\lambda|_{\mathbb{G}_\pm}$). This indicates that $\bar{\gamma}$ should play a critical role (notice that $\lambda+\bar{\gamma}\cdot\delta_0$ cannot induce a scale function if $\bar{\gamma}>0$). 
Indeed, we shall build a phase transition in terms of $\bar{\gamma}$ for this stiff problem in Theorem~\ref{THM45}: 
\begin{itemize}
\item[(1)]  $\bar{\gamma}=\infty$: The flow cannot cross the singular barrier and the conduction is divided into two separate parts. Mathematically, $X^\varepsilon$ converges to a non-irreducible diffusion, namely a union of two separate reflecting diffusions on $[0+,\infty)$ and $(-\infty, 0-]$ respectively. 
\item[(2)] $0<\bar{\gamma}<\infty$: This is the most interesting case. The flow could penetrate the singular barrier partially, and in the probabilistic counterpart, penetrations are realized by additional jumps between $0+$ and $0-$.  
\item[(3)] $\bar{\gamma}=0$: The barrier makes no sense and $X^\varepsilon$ converges to the diffusion associated with \eqref{EQ1FFL}. 
\end{itemize}
We call the three patterns of thermal conduction above the \emph{impermeable pattern} for $\bar{\gamma}=\infty$, \emph{semi-permeable pattern} for $0<\bar{\gamma}<\infty$ and \emph{permeable pattern} for $\bar{\gamma}=0$ respectively. Particularly, the Brownian case with the conductivity \eqref{EQ1AXA} is such that $m=\lambda$ is the Lebesgue measure and 
\[
	\gamma_\varepsilon(dx)=\frac{1}{a_\varepsilon(x)}dx=(\kappa\varepsilon)^{-1}dx.
\]
As a consequence, $\bar{\gamma}=2/\kappa$ and the parameter $\kappa$ is nothing but the reciprocal of total thermal resistance.

\begin{figure}
\centering
\includegraphics[scale=0.85]{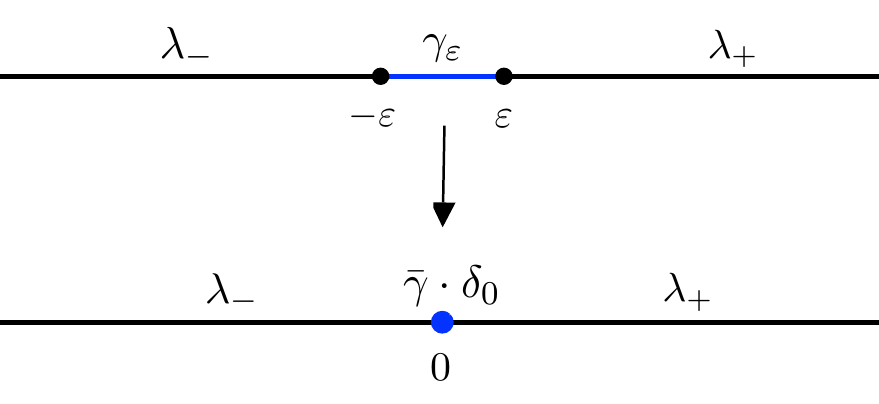}
\caption{Stiff problem in the manner of thermal resistance.}\label{FIG1}
\end{figure}


This paper is organized as follows. In \S\ref{SEC2}, we shall review several transforms of Markov processes and their counterparts in theory of Dirichlet form for later use.  In \S\ref{SEC3}, we shall extend the SNOB to the so-called snapping out Markov process on a general state space. This notion is the probabilistic counterpart of semi-permeable pattern in stiff problem. It is, by definition, a Markov process obtained by transforms of killing and piecing out with respect to the same finite measure. We shall derive the Dirichlet form of snapping out Markov process in Theorem~\ref{THM35} and explore its properties. Particularly, rich facts about SNOB are presented in Proposition~\ref{PRO310}. 
In \S\ref{SEC34}, several other examples of snapping out Markov processes are raised. 

The section \S\ref{SEC4} is devoted to the general stiff problems in one-dimensional space. As said above, the phase transitions are built in Theorem~\ref{THM45}. The convergences of Dirichlet forms in this theorem are in the sense of Mosco. Mosco convergence is reviewed in \S\ref{SEC41} and particularly, it implies the convergence of associated Markov processes in the sense of finite dimensional distributions as stated in Corollary~\ref{COR49}. We shall derive the generators of three Markov processes related to every phase transition in Proposition~\ref{PRO45}. The characterizations of their generators play important roles in studying the boundary conditions of flux at $0$ in the three patterns of thermal conduction. In \S\ref{SEC44}, we find that  the flux is continuous at $0$ in permeable pattern and satisfies the boundary conditions \eqref{EQ4UUO} and \eqref{EQ4AUA} at $0$ in impermeable and semi-permeable patterns respectively. 

\subsection*{Notations}
Let us put some often used notations here for handy reference, though we may restate their definitions when they appear.

Given a topological space $E$, $\mathcal{B}(E), \mathcal{B}_b(E), C(E), C_b(E)$ and $C_c(E)$ are families of all the Borel measurable functions, all the bounded Borel measurable functions, all the continuous functions, all the bounded continuous functions and all the continuous functions with compact supports on $E$ respectively. For an interval $I$,  the classes $C_c(I), C^1_c(I)$ and $C^\infty_c(I)$ denote the spaces of all the continuous functions with compact supports, all the continuously differentiable functions with compact supports and all the infinitely differentiable functions with compact supports on $I$ respectively.

The notation `$:=$' is read as `to be defined as'. For two functions $f,g$ and a measure $\mu$ on $E$, $(f,g)_{\mu}:=\int_E fgd\mu$ and $\langle f,\mu\rangle:=\int_E fd\mu$.
Notation $dx$ stands for the Lebesgue measure on $\mathbb{R}$ or an interval throughout the paper. For $x\in\mathbb{R}^d$, $|x|$ is the Euclidean norm of $x$.  
The restrictions of a measure $\mu$ and a function $f$ to $I$ are denoted by $\mu|_I$ and $f|_I$ respectively. Given two measures $\mu$ and $\nu$, $\mu\ll \nu$ means $\mu$ is absolutely continuous with respect to $\nu$. 
 
Given a scale function $\ss$, namely a continuous and strictly increasing function, on $I$, $d\ss$ represents its induced Lebesgue-Stieltjes measure on $I$. We also use $\lambda$ for $d\ss$. For a function $f$ on $I$, $f\ll \ss$ (or $f\ll \lambda$) means $f=g\circ \ss$ for some absolutely continuous function $g$ and $\frac{df}{d\ss}=\frac{df}{d\lambda}:=g'\circ \ss$. 

For any function $u$ (resp. a measure $\nu$) on $\mathbb{G}=\mathbb{G}_-\cup \mathbb{G}_+=(-\infty, 0-]\cup [0+,\infty)$, $u_+:=u|_{\mathbb{G}_+}$ and $u_-:=u|_{\mathbb{G}_-}$ (resp. $\nu_+:=\nu|_{\mathbb{G}_+}$ and $\nu_-:=\nu|_{\mathbb{G}_-}$). The subscript `$\pm$' is read as `$+$ and $-$'. For example, $u_\pm\ll \nu_\pm$ means $u_+\ll \nu_+$ and $u_-\ll \nu_-$. 

\section{Transforms of Markov processes}\label{SEC2}

In this section we shall review several transforms of Markov processes, which will be frequently used in the subsequent sections. 
Let $E$ be a locally compact separable metric space and $m$ a positive Radon measure fully supported on $E$. The one-point compactification of $E$ is written as $E_\Delta:=E\cup\Delta$ (if $E$ is compact then $\Delta$ is attached as an isolated point). Further let $(\EE,\FF)$ be a regular Dirichlet form on $L^2(E,m)$ associated with an $m$-symmetric Markov process $X=\left(\Omega,\mathcal{F},X_t,\mathcal{F}_t,\theta_t, \zeta, (\mathbf{P}_x)_{x\in E_\Delta}\right)$ on $E$. The extended Dirichlet space of $(\EE,\FF)$ is denoted by $\FF_\mathrm{e}$. Every function in a Dirichlet space will be taken to be its quasi-continuous version for convenience. All the terminologies above are standard, and we refer them to \cite{CF12, FOT11}.

\subsection{Killing transform}\label{SEC21}

The first transform is called the killing transform.  It kills the trajectories according to a given tactic and attains a new Markov process. The concrete description is referred to \cite[Chapter III]{BG68}. In the following, we shall present its counterpart in theory of Dirichlet form. 

Let $\mu$ be a smooth Radon measure with respect to $(\EE, \FF)$, which means $\mu$  charges no $\EE$-polar set. The perturbed Dirichlet form by $\mu$ is given by (Cf. \cite[\S6.1]{FOT11})
\begin{equation}\label{DFE}
\begin{aligned}
  \FF^\mu &=\FF\cap L^2(E,\mu),\\
 \EE^\mu(f,g)&=\EE(f,g)+\int_E fgd\mu,\quad f,g\in \FF^\mu.
 \end{aligned}
\end{equation}
It is also a regular Dirichlet form on $L^2(E, m)$ in the light of \cite[Theorem~5.1.6]{CF12}. 

The associated Markov process of $(\EE^\mu,\FF^\mu)$, denoted by $X^\mu=(X^\mu_t)_{t\geq 0}$, is nothing but the subprocess of $X$ induced by a multiplicative functional $\left(\mathrm{e}^{-A_t}\right)_{t\geq 0}$ (Cf. \cite{BG68}), where $(A_t)_{t\geq 0}$ is the positive continuous additive functional (PCAF in abbreviation) of $\mu$ in the Revuz correspondence. 
Roughly speaking, the trajectories of $X^\mu$ are realized from those of $X$ by killing at some rates depending on $\mu$.  
Particularly, the semigroup $P^\mu_t$ of $X^\mu$ can be written as
\[
	P_t^\mu f(x)=\mathbf{E}_x\left[ \mathrm{e}^{-A_t}f(X_t) \right]
\]
for any positive function $f$. 

\subsection{Time change}\label{TRM}

The second transform is the time change. Take a PCAF $(A_t)_{t\geq 0}$ of $X$ with $\mu$ being its Revuz measure. Denote the quasi support (Cf. \cite{CF12}) of $\mu$ by $F$. The right continuous inverse $\tau_t$ of $A_t$ is defined by
\[
	\tau_t(\omega):=
	\begin{cases}
		\inf\{s: A_s(\omega>t\},\quad &\text{if }t<A_{\zeta(\omega)-}(\omega),\\
		\infty,\quad &\text{if }t\geq A_{\zeta(\omega)-}(\omega). 
	\end{cases}	
\]
Set
\[
\check{X}_t(\omega):=X_{\tau_t(\omega)}(\omega),\quad \check{\zeta}(\omega):=A_{\zeta(\omega)-}(\omega). 
\]
Then $\check{X}=(\check{X}_t, \check{\zeta}, (\mathbf{P}_x)_{x\in F_\Delta})$ is a right process on $F$ and called the \emph{time-changed process} of $X$ by the PCAF $A$ or speed measure $\mu$.  

The counterpart of time-changed process in theory of Dirichlet form is the so-called \emph{trace Dirichlet form}. Its idea goes back to Douglas \cite{D31} from an analytic viewpoint, and Chen et al. studied the traces of general symmetric Dirichlet forms in \cite{CFY06}. In fact, the time-changed process $\check{X}$ is a $\mu$-symmetric Markov process on $F$. Its associated Dirichlet form on $L^2(F,\mu)$ is actually the trace Dirichlet form of $(\EE,\FF)$ on $F$ and given by
\begin{equation}\label{EQ2FFE}
\begin{aligned}
	\check{\FF}&=\FF_\mathrm{e}|_F \cap L^2(F,\mu), \\
	\check{\EE}(u|_F,v|_F)&=\EE(\mathbf{H}_Fu,\mathbf{H}_Fv),\quad \forall u,v \in \FF_\mathrm{e}|_F, 
\end{aligned}
\end{equation}
where $\mathbf{H}_Fu(x):=\mathbf{E}_x[u(X_{\sigma_F}), \sigma_F<\infty]$ and $\sigma_F:=\inf\{t>0: X_t\in F\}$ is the hitting time of $F$. If $\mu$ is Radon, then $(\check{\EE},\check{\FF})$ is regular. We refer further considerations of time-changed processes and trace Dirichlet forms to \cite{CF12, CFY06}.

\subsection{Darning}\label{D}

The transform of darning was first performed in \cite{CF08} to study the one-point extensions of Markov process.   
Following \cite{CP17}, let $K_1, K_2,\dots, K_n$ be disjoint compact subsets of $E$ with positive capacity. Denote $D=E\setminus\cup^{n}_{i=1}K_i$, and short each $K_i$ into a single point $a^*_i$. Set a measure $m^*$ on $E^*:=D\cup\{a^*_1, a^*_2, \dots, a^*_n\}$ by letting $m^*=m$ on $D$ and $m^*(\{a^*_1, a^*_2, \dots, a^*_n\})=0$. The \emph{Markov process with darning} induced by $X$ is a strong Markov process $X^*$ on $E^*$ such that
\begin{itemize}
\item[(1)] the part process of $X^*$ in $D$ has the same law as the part process of $X$ in $D$;
\item[(2)] the jumping measure and killing measure of $X^*$ have the property inherited from $X$ without additional jumps or killings. 
\end{itemize}   
It is shown in \cite{CP17} that such a process exists and is unique in law, and its Dirichlet form $(\EE^*,\FF^*)$ is given by 
\begin{equation}\label{DR}
 \begin{aligned}
 & \FF^*=\left\{f^*: f\in \FF, f \mbox{ is constant $\EE$-q.e. on each $K_j$}\right\},\\
 & \EE^*(f^*,g^*)=\EE(f,g),\quad f^*,g^*\in \FF^*,
 \end{aligned}
\end{equation}
where $f^*(x):=f(x)$ for $x\in D$ and $f^*(a_i^*):=f(y)$ with $y\in K_i$. 
Moreover, $(\EE^*,\FF^*)$ is a regular Dirichlet form on $L^2(E^*,m^*)$ by \cite[Theorem~3.3]{CP17}. 

\subsection{Piecing out}
Piecing out transform raised by Ikeda et al. in \cite{INW66} is, in some sense, an inverse transform of killing. As in \cite{INW66}, let $W:=\Omega \times E$ with $\mathcal{B}(W):=\mathcal{F}\otimes \mathcal{B}(E)$ and for any $\mathtt{w}=(\omega, y)\in W$, set
\begin{equation}\label{EQ2XTW}
\dot{X}_t(\mathtt{w}):=\left\lbrace
\begin{aligned}
	X_t(\omega),\quad t<\zeta(\omega), \\
	y,\quad\quad t\geq \zeta(\omega). 
\end{aligned}\right.
\end{equation}
Take an appropriate kernel $\nu(\omega, dy)$ on $\Omega\times E_\Delta$ with $\nu(\omega, \cdot)$ being a probability measure on $E_\Delta$, and for each $x\in E_\Delta$ put a probability measure $\mathbf{Q}_x(d\mathtt{w}):=\mathbf{P}_x(d\omega)\nu(\omega, dy)$ on $W$. Further let $(\tilde{\Omega}, \tilde{\mathcal{F}})$ be the product of an infinite, countable copies of $(W,\mathcal{B}(W))$. Clearly, there exists a unique probability measure $\tilde{\mathbf{P}}_x$ on $(\tilde{\Omega}, \tilde{\mathcal{F}})$ such that
\[
\tilde{\mathbf{P}}_x[ d\mathtt{w}_1,\cdots,d\mathtt{w}_n]
 =\mathbf{Q}_x[d\mathtt{w}_1]\mathbf{Q}_{y_1}[d\mathtt{w}_2]\cdots\mathbf{Q}_{y_{n-1}}[d\mathtt{w}_n],
\]
where $\mathtt{w}_i=(\omega_i, y_i)$ for $1\leq i\leq n$. Define a new trajectory for $\tilde{\mathtt{w}}=(\mathtt{w}_1,\cdots, \mathtt{w}_n,\cdots)\in \tilde{\Omega}$ as follows:
\[
\tilde{X}_t(\tilde{\mathtt w})= \begin{cases}
 \dot{X}_t(\mathtt{w}_1), &  \mbox{if}~0\leq t\leq \zeta(\omega_1), \\
 \cdots & \\
 \dot{X}_{t-(\zeta(\omega_1)+\ldots+\zeta(\omega_n))}(\mathtt{w}_{n+1}), & \mbox{if}~\sum\limits_{i=1}^n\zeta(\omega_i)<t\leq \sum\limits_{i=1}^{n+1}\zeta(\omega_i),\\
 \cdots & \\
 \Delta & \mbox{if}~t\geq\tilde{\zeta}(\tilde{\mathtt{w}}):= \sum\limits_{i=1}^{N(\tilde{\mathtt{W}})}\zeta(\omega_{i}),
 \end{cases} 
\]
where $N(\tilde{\mathtt{w}})=\inf\{i:\zeta(\omega_i)=0\}$ with $\inf\emptyset:=\infty$.
After defining the shift operators $\tilde{\theta}_t$ and filtration $\tilde{\mathcal{F}}_t$ on $\tilde{\Omega}$ accordingly, the principal result of \cite{INW66} tells us 
\begin{equation}\label{EQ2XOF}
	\tilde{X}=\left(\tilde{\Omega},\tilde{\mathcal{F}},\tilde{X}_t,\tilde{\mathcal{F}}_t,\tilde{\theta}_t, \tilde{\zeta}, (\tilde{\mathbf{P}}_x)_{x\in E_\Delta}\right)
\end{equation}
is a right continuous Markov process on $E_\Delta$ with $\tilde{\mathbf{P}}_{\Delta}[\tilde{X}_t=\Delta, \forall t\geq 0]=1$. Intuitively speaking, $\tilde{X}$ is realized by  resurrection after the death of $X$, and more precisely, it takes a random reborn site  according to $\nu$ and continues the motion along a new trajectory of $X$ starting from this reborn site until the next death. 
The kernel $\nu$ is called the \emph{instantaneous distribution} of piecing out transform in \cite{INW66}. In this paper, we shall take a special form of instantaneous distribution as follows.
\begin{definition}
Let
\begin{equation}\label{EQ2NDY}
\nu(\omega, dy):=\left\lbrace 
	\begin{aligned}
	\nu^{\#}(dy),\;\;\quad &\text{for }X_{\zeta(\omega)-}(\omega)\in E,\\
	\delta_{\{\Delta\}}(dy),\quad &\text{for } X_{\zeta(\omega)-}(\omega)=\Delta
	\end{aligned}
	\right.
\end{equation}
with some probability measure $\nu^{\#}$ on $E$. In abuse of terminology, we call \eqref{EQ2XOF} the \emph{piecing out process} with instantaneous distribution $\nu^{\#}$ induced by $X$.
\end{definition}

The choice of $\nu$ in \eqref{EQ2NDY} indicates that the left limit $\tilde{X}_{t-}$ exists in $E$ for any $t<\tilde{\zeta}$. This is necessary for $\tilde{X}$ to be a Hunt process. Furthermore, we can conclude the following lemma by \cite{INW66}.

\begin{lemma}\label{LM21}
Let $\tau(\tilde{\mathtt{w}}):= \zeta(\omega_1)$ for $\tilde{\mathtt w}=(\mathtt{w}_1,\mathtt{w}_2,\cdots)\in \tilde{\Omega}$ and $\mathtt{w}_i=(\omega_i,y_i)$. Then $\tau$ is an $\tilde{\mathcal{F}}_t$-stopping time. 
\end{lemma}

\section{Snapping out Markov processes}\label{SEC3}

Lejay raised a model which he called a \emph{snapping out Brownian motion} (abbrviated in SNOB) in \cite{L16}. It was introduced for the probabilistic description of a stiff problem in one-dimensional space. In this section, we shall first recall the main ideas of this model, and then extend this notion to the so-called \emph{snapping out Markov process} on a general state space. This class of Markov processes will be used in \S\ref{SEC4} to characterize the semi-permeable patterns of thermal conductions in stiff problems. 

\subsection{Snapping out Brownian motion: Lejay's approach}

Let $\mathbb{G}:=(-\infty,0-]\cup[0+,\infty)$, where $0$ in $\mathbb{R}$ corresponds to either $0+$ or $0-$ viewed as two distinct points. In other words, $\mathbb{G}$ is composed of two connected components, say $(-\infty,0-]$ and $[0+,\infty)$. Write
\[
	\mathbb{G}_+:=[0+,\infty),\quad \mathbb{G}_-:=(-\infty, 0-].
\]
An SNOB is a Markov process living in $\mathbb{G}$. 
Precisely, let us start with a reflecting Brownian motion $R^+=(R^+_t)_{t\geq 0}$ on $\mathbb{G}_+$. Denote its local time at $0+$  by $(L^+_t)_{t\geq 0}$. Namely, 
\[
L^+_t=\lim_{\epsilon\downarrow 0}\frac{1}{2\epsilon}\int_{0}^t 1_{[0+,\epsilon)}(R^+_s)ds,\quad t\geq 0
\]
is a PCAF of $R^+$ with $\frac{1}{2}\delta_{\{0+\}}$ being its Revuz measure.
Let $\xi$ be an exponential random variable with a parameter $\kappa>0$ independent of $R^+$. Set
\[
Z^+_t:=\left\lbrace \begin{aligned}
 R^+_t,\quad & \mbox{if}~t<\mathfrak{t}:=\inf\{t:L^+_t>\xi\}; \\
 \Delta,\quad & \mbox{if}~t\geq \mathfrak{t}  	
 \end{aligned}
  \right. 
\]
with $\Delta$ being the trap as usual. Then $Z^+=(Z^+_t)_{t\geq 0}$ is called the \emph{elastic Brownian motion} on $\mathbb{G}_+$. We extend $Z^+$ to a process $Z$ on $\mathbb{G}$ by symmetry and call $Z$ the elastic Brownian motion on $\mathbb{G}$. In \cite{L16}, the author introduced the following definition of SNOB by means of this elastic Brownian motion and the piecing out transform.

\begin{definition}[\cite{L16}]\label{DEF31}
Let $Z$ be the elastic Brownian motion with the parameter $\kappa>0$ on $\mathbb{G}$. Then the piecing out process with instantaneous distribution $\frac{1}{2}\left(\delta_{\{0+\}}+\delta_{\{0-\}}\right)$ induced by $Z$ is called the snapping out Brownian motion on $\mathbb{G}$.
\end{definition}

Intuitively, we may think of the local time $L^+$ as the `hitting intensity' at the boundary $0+$, which increases once $R^+$ encounters $0+$. When the hitting intensity is overloaded, i.e. the local time is greater than the given threshold $\xi$, the elastic Brownian motion will die, while the SNOB will be reborn at $0+$ or $0-$ with equal probability.


\subsection{Snapping out Markov processes}\label{SEC32}
Throughout this part, $E$ is taken to be a locally compact separable metric space and $m$ is a Radon measure fully supported on it.
Inspired by the SNOB, we introduce the so-called \emph{snapping out Markov process} on a general state space as follows. 

\begin{definition}\label{DEF32}
Let $X=(X_t)_{t\geq 0}$ be an $m$-symmetric Markov process on $E$ associated with a regular Dirichlet form $(\EE,\FF)$ on $L^2(E,m)$, and take a positive, finite smooth measure $\mu$ on $E$. Denote the subprocess of $X$ induced by $\mu$ by $X^\mu=(X^\mu_t)_{t\geq 0}$ and set $\mu^{\#}:=\mu/\mu(E)$. Then the piecing out process, denoted by $X^\mathrm{s}=(X^\mathrm{s}_t)_{t\geq 0}$, with instantaneous distribution $\mu^{\#}$ induced by $X^\mu$ is called the \emph{snapping out Markov process} with respect to $X$ and $\mu$. 
\end{definition}

We need to emphasize that the Revuz correspondence between $\mu$ and the associated PCAF depends on the symmetric measure $m$. So the killing transform in Definition~\ref{DEF32} is also relevant to $m$. See Example~\ref{EXA311} for further discussions. 
 
\begin{remark}\label{RM33}
In Definition~\ref{DEF31}, the construction of SNOB starts with a two-sided reflecting Brownian motion $R=(R_t)_{t\geq 0}$ on $\mathbb{G}$ (more precisely, a union of two separate reflecting Brownian motions on $\mathbb{G}_+$ and $\mathbb{G}_-$ respectively). It is  not difficult to find that this two-sided reflecting Brownian motion is symmetric with respect to the Lebesgue measure on $\mathbb{G}$ and its associated Dirichlet form is regular on $L^2(\mathbb{G})$. Moreover, the two-sided elastic Brownian motion $Z$ is actually the subprocess of $R$ induced by $\frac{\kappa}{2}\left(\delta_{\{0+\}}+\delta_{\{0-\}} \right)$. 
\end{remark}


In advance of presenting the principal result of this part, we need to prepare some notations. Let $\zeta, \zeta^\mu, \zeta^\mathrm{s}$ (resp. $P_t, P^\mu_t, P_t^\mathrm{s}$ and $R_\alpha, R_\alpha^\mu, R_\alpha^\mathrm{s}$) be the lifetimes (resp. semigroups and resolvents) of $X, X^\mu, X^\mathrm{s}$ respectively. In abuse of notations, we use the same symbol for the expectations of $X, X^\mu, X^\mathrm{s}$. For example, 
\[
	P^{\dagger}_tf(x)=\mathbf{E}_xf(X^{\dagger}_t), \quad R^\dagger_\alpha f(x)=\mathbf{E}_x\int_0^\infty\mathrm{e}^{-\alpha t}f(X^\dagger_t)dt,
\]
where $\dagger$ is vacant or stands for $\mu$ or $\mathrm{s}$.  
 The Dirichlet form of $X^\mu$ on $L^2(E,m)$ is given by \eqref{DFE}. 
 Accordingly, we can also write down $P^\mu_t$ and $R^\mu_\alpha$ by using $X$ (Cf. \cite{FOT11}). Moreover, the following lemma links the resolvents of $X^\mu$ and $X^\mathrm{s}$. Note that $|\mu|:=\mu(E)$. 

\begin{lemma}\label{LM34}
For $\alpha>0$ and any non-negative function $f$, it holds that
\begin{equation}\label{R1}
	R^\mathrm{s}_\alpha f=R^\mu_\alpha f+\frac{\langle R^\mathrm{s}_\alpha f, \mu\rangle}{|\mu|}\cdot \mathbf{E}_x\left[\mathrm{e}^{-\alpha \zeta^\mu}; X^\mu_{\zeta^\mu-}\in E\right].
\end{equation}
\end{lemma}
\begin{proof}
We first note that $\zeta^\mu$ is a stopping time of $X^\mathrm{s}$ in the sense of Lemma~\ref{LM21}, and $X^\mathrm{s}=X^\mu$ before $\zeta^\mu$. Since $X^\mu_{\zeta^\mu-}=\Delta$ implies $\zeta^\mathtt{s}=\zeta^\mu$, it follows that
\[
\begin{aligned}
R_\alpha^\mathrm{s}f(x)&=\mathbf{E}_x\int_0^{\infty}\mathrm{e}^{-\alpha t}f(X^\mathrm{s}_t)dt\\& =\mathbf{E}_x\int_0^{\zeta^\mu}\mathrm{e}^{-\alpha t}f(X^\mathrm{s}_t)dt +\mathbf{E}_x\left[\int_{\zeta^\mu}^{\infty}\mathrm{e}^{-\alpha t}f(X^\mathrm{s}_t)dt; X^\mu_{\zeta^\mu-}\in E \right]\\
&=R_\alpha^\mu f(x)+\mathbf{E}_x \left[\mathrm{e}^{-\alpha \zeta^\mu} \cdot  \mathbf{E}_x\left[\left(\int_0^\infty\mathrm{e}^{-\alpha t}f(X^\mathrm{s}_t)dt\right)\circ \theta^{\mathrm{s}}_{\zeta^\mu}\bigg|\mathcal{F}^{\mathrm{s}}_{\zeta^\mu}\right]; X^\mu_{\zeta^\mu-}\in E \right] \\
&=R_\alpha^\mu f(x)+\mathbf{E}_x \left[\mathrm{e}^{-\alpha \zeta^\mu} \cdot R^\mathrm{s}_\alpha f(X^\mathrm{s}_{\zeta^\mu}) ; X^\mu_{\zeta^\mu-}\in E \right]. 
\end{aligned}\]
On the other hand, $X^\mathrm{s}_{\zeta^\mu}$ is distributed as $\mu^{\#}$ and independent of $\zeta^\mu$ and $X^\mu$ by \eqref{EQ2XTW} and \eqref{EQ2NDY}. Then we can conclude \eqref{R1}. 
That completes the proof. 
\end{proof}

Now we have a position to present the principal theorem of this part. It tells us if $X$ has no killing inside, then the snapping out Markov process $X^\s$ is $m$-symmetric and the associated Dirichlet form can be also characterized.

\begin{theorem}\label{THM35}
Let $X$ and $\mu$ be in Definition~\ref{DEF32} and $X^\mathrm{s}$ be the snapping out Markov process with respect to $X$ and $\mu$. Set $|\mu|=\mu(E)$. Assume that $X$ or $(\EE,\FF)$ has no killing inside. 
Then $X^\mathrm{s}$ is $m$-symmetric on $E$, and its associated Dirichlet form is regular on $L^2(E,m)$ and given by 
\begin{equation}\label{EQ3FSU}
 \begin{aligned}
  \FF^\mathrm{s}&=\left\{u\in \FF: \int_{E\times E}\left(u(x)-u(y)\right)^2\mu(dx)\mu(dy)<\infty\right\}, \\
 \EE^\mathrm{s}(u,v)&=\EE(u,v)+\frac{1}{2|\mu|}\int_{E\times E}\left(u(x)-u(y)\right)\left(v(x)-v(y)\right)\mu(dx)\mu(dy),\quad u,v\in \FF^\s.
 \end{aligned}
\end{equation}
Furthermore, any special standard core of $(\EE,\FF)$ remains to be a special standard core of $(\EE^\mathrm{s},\FF^\mathrm{s})$. 
\end{theorem}
\begin{proof}
We first show $(\EE^\mathrm{s},\FF^\mathrm{s})$ given by \eqref{EQ3FSU} is a regular Dirichlet form on $L^2(E,m)$. It is proved in \cite{AS93} that \eqref{EQ3FSU} is a Dirichlet form. Thus we need only prove its regularity. Let $\mathscr{C}$ be a special standard core of $(\EE,\FF)$. Then it is also a core of $(\EE^\mu, \FF^\mu)$ by \cite[Theorem~5.1.6]{CF12}. Denote the families of all the bounded functions in $\FF$, $\FF^\mu$ and $\FF^\mathrm{s}$ by $\FF_\mathrm{b}$, $\FF^\mu_\mathrm{b}$ and $\FF^\mathrm{s}_\mathrm{b}$ respectively. Since $\mu$ is a finite measure, we have
\[
	\mathscr{C}\subset \FF^\mathrm{s}_\mathrm{b}=\FF_\mathrm{b}=\FF^\mu_\mathrm{b}.
\]
On the other hand, for any $u\in \FF^\mathrm{s}_\mathrm{b}=\FF^\mu_\mathrm{b}$, 
\[
\begin{aligned}
	\EE^\mathrm{s}_1(u,u)&=\EE_1(u,u)+\frac{1}{2|\mu|}\int \left(u(x)-u(y) \right)^2\mu(dx)\mu(dy) \\
	 &= \EE_1(u,u)+\int u^2d\mu-\langle u, \mu\rangle^2/|\mu|  \\
	 &\leq \EE^\mu_1(u,u).
\end{aligned}\]
For any $u\in \FF^\mathrm{s}_\mathrm{b}=\FF^\mu_\mathrm{b}$, we can take a sequence $\{u_n:n\geq 1\}$ in $\mathscr{C}$ such that $u_n$ converges to $u$ in $\EE^\mu_1$-norm. Thus from the above inequality, we can obtain that $u_n$ also converges to $u$ in $\EE^\mathrm{s}_1$-norm. This implies $(\EE^\mathrm{s},\FF^\s)$ is a regular Dirichlet form on $L^2(E,m)$ and $\mathscr{C}$ is its special standard core.

Next, we assert that $X^\s$ is $m$-symmetric under the assumption that $X$ has no killing inside. Note that $\mu$ is a measure of finite energy integral with respect to $\EE^\mu$, i.e. 
\[
	\mu(|v|)\leq |\mu|\sqrt{\EE^\mu_1(v,v)},\quad v\in \FF^\mu.
\]
Thus the $\alpha$-potential $U^\mu_{\alpha}\mu$ of $\mu$ exists with
\begin{equation}\label{EQ3EMU}
\EE^\mu_{\alpha}(U^\mu_{\alpha}\mu, v)=\langle v,\mu\rangle,\quad v\in \FF^\mu.
\end{equation}
Since $X$ has no killing inside, the killing measure of $(\EE^\mu,\FF^\mu)$ is equal to $\mu$. Applying \cite[Lemma~4.5.2]{FOT11} to $\mu$, we have
\[
	U^\mu_{\alpha}\mu(\cdot)=\mathbf{E}_\cdot \left[\mathrm{e}^{-\alpha \zeta^\mu}; X^\mu_{\zeta^\mu-}\in E\right].
\]
Clearly, $\mathbf{P}_x\left[\zeta^\mu=0;X^\mu_{\zeta^\mu-}\in E \right]=0$ and this implies $\langle U^\mu_\alpha \mu, \mu\rangle<|\mu|$ for any $\alpha>0$. For any positive function $f$, it follows from Lemma~\ref{LM34} that
\[
\langle R^\s_\alpha f,\mu\rangle=\langle R^\mu_\alpha f,\mu\rangle+\frac{\langle R^\s_\alpha f,\mu\rangle \cdot \langle U^\mu_\alpha \mu, \mu\rangle}{|\mu|}
\]
and thus 
\[\langle R^\s_\alpha f,\mu\rangle=\frac{\langle R^\mu_\alpha f, \mu\rangle}{1-\langle U^\mu_{\alpha} \mu, \mu\rangle/|\mu|}. \]
We substitute this into \eqref{R1} to find that
\begin{equation}\label{EQ3RSF}
	R^\mathrm{s}_\alpha f=R^\mu_\alpha f+\frac{\langle R^\mu_\alpha f, \mu\rangle}{1-\langle U^\mu_{\alpha} \mu, \mu\rangle/|\mu|} \cdot U^\mu_\alpha \mu.
\end{equation}
For another positive function $g\in L^2(E,m)$, \eqref{EQ3EMU} implies 
\begin{equation}\label{EQ3UAM}
	(U_\alpha^\mu \mu, g)_m=\EE^\mu_\alpha(U^\mu_\alpha\mu, R^\mu_\alpha g)=\langle R^\mu_\alpha g, \mu\rangle. 
\end{equation}
Then from \eqref{EQ3RSF} we have $(R^\mathrm{s}_\alpha f, g)_m=(f,R^\mathrm{s}_\alpha g)_m$. This concludes that $X^\s$ is $m$-symmetric.

Finally, it suffices to prove that for any $\alpha>0$ and $f\in L^2(E,m)$,
\[
R^\s_\alpha f\in \FF^{\s},\qquad \EE^\mathrm{s}_\alpha(R^\s_\alpha f, g)=(f,g)_m,\quad \forall g\in \FF^\mathrm{s}_\mathrm{b}. 
\]
Note that $\FF^\mu\subset \FF^\mathrm{s}$. Since $R^\mu_\alpha f, U^\mu_\alpha \mu \in \FF^\mu$, it follows from \eqref{EQ3RSF} that $R^\s_\alpha f\in \FF^\mu\subset \FF^\mathrm{s}$. Moreover, we can obtain from Lemma~\ref{LM34}, \eqref{EQ3EMU} and $g\in \FF^\mathrm{s}_\mathrm{b}=\FF^\mu_\mathrm{b}$ that
\[
\begin{aligned}
	\EE^\mathrm{s}_\alpha(R^\s_\alpha f, g)&=\EE^\mu_\alpha(R^\s_\alpha f, g)-\frac{1}{|\mu|}\langle R^\s_\alpha f, \mu\rangle \cdot \langle g, \mu\rangle \\
&=\EE^\mu_\alpha(R^\mu_\alpha f, g)+\frac{\langle R^\s_\alpha f, \mu\rangle}{|\mu|}\cdot \EE^\mu_\alpha(U^\mu_\alpha \mu, g)-\frac{1}{|\mu|}\langle R^\s_\alpha f, \mu\rangle \cdot \langle g, \mu\rangle \\
&=(f,g)_m+\frac{1}{|\mu|}\langle R^\s_\alpha f, \mu\rangle \cdot \langle g, \mu\rangle-\frac{1}{|\mu|}\langle R^\s_\alpha f, \mu\rangle \cdot \langle g, \mu\rangle \\
&=(f,g)_m.
\end{aligned}\]
That completes the proof. 
\end{proof}

The assumption that $X$ has no killing inside is necessary for the symmetry of $X^\s$. For interpreting this fact, suppose the killing measure $k$ ($\neq 0$) of $X$ is of finite energy integral with respect to $(\EE,\FF)$. Then $\mathbf{E}_x[\mathrm{e}^{-\alpha \zeta^\mu}; X^\mu_{\zeta^\mu-}\in E]=U^\mu_\alpha(k+\mu)(x)$. Mimicking \eqref{EQ3RSF}, we can conclude that for positive functions $f$ and $g$,
\[
	(R^\s_\alpha f, g)_m=(R^\mu_\alpha f,g)_m+\frac{\langle R^\mu_\alpha f, \mu\rangle\cdot \langle R^\mu_\alpha g, k+\mu\rangle}{1-\langle U^\mu_\alpha(k+\mu),\mu\rangle/|\mu|}.
\]
Consequently, the presence of $k$ breaks the symmetry of $X^\s$. At a heuristic level, many jumps are added into the trajectories by the piecing out transform. The additional jumps start with an initial `distribution' given by the killing measure $k+\mu$ and arrive at random sites distributed as $\mu^{\#}$. When $k\neq 0$, the additional jumping measure is not symmetric and thus the symmetry of $X^\s$ is broken. 

On the other hand, the regularity of \eqref{EQ3FSU} only depends on the finiteness and smoothness of $\mu$. Even if $k\neq 0$, \eqref{EQ3FSU} is still regular and corresponds to an $m$-symmetric Markov process. This process could be realized as follows: We first construct the resurrected Markov process $X^\mathrm{res}$ of $X$ according to \cite[Theorem~5.2.17]{CF12}, then attain the snapping out Markov process $X^\mathrm{res, s}$ with respect to $X^\mathrm{res}$ and $\mu$ and finally apply the killing transform induced by $k$ to $X^\mathrm{res, s}$. Note that $k$ is also smooth with respect to $X^\mathrm{res, s}$ due to the following corollary. 

\begin{corollary}\label{COR36}
Let $X,\mu$ be in Theorem~\ref{THM35}, but we do not assume $X$ has no killing inside. Further let $(\EE^\mathrm{s},\FF^\mathrm{s})$ be defined by \eqref{EQ3FSU}. Then $(\EE^\mathrm{s},\FF^\mathrm{s})$ is a regular Dirichlet form on $L^2(E,m)$ sharing the same set of quasi-notions with $(\EE,\FF)$. In other words, an increasing sequence of closed subsets of $E$ (resp. a subset of $E$, a function on $E$) is an $\EE^\mathrm{s}$-nest (resp. $\EE^\mathrm{s}$-polar set, $\EE^\mathrm{s}$-quasi-continuous function) if and only it is an $\EE$-nest (resp. $\EE$-polar set, or $\EE$-quasi-continuous function).
\end{corollary}
\begin{proof}
We need only prove that an $\EE^\mathrm{s}$-nest is an $\EE$-nest and vice versa. Note that $(\EE^\mu,\FF^\mu)$ has the same quasi-notions as $(\EE,\FF)$ due to \cite[Theorem~5.1.4]{CF12}. Denote the $1$-capacities of $(\EE,\FF)$, $(\EE^\mu, \FF^\mu)$ and $(\EE^\mathrm{s},\FF^\mathrm{s})$ by $\text{Cap}$, $\text{Cap}^\mu$ and $\text{Cap}^\s$ respectively. Since
\begin{equation}\label{EQ3FSF}
	\FF^\mathrm{s}\subset \FF,\quad \EE^\mathrm{s}(u,u)\geq \EE(u,u),\quad \forall u\in \FF^\mathrm{s}, 
\end{equation}
it follows that $\text{Cap}(A)\leq \text{Cap}^\s(A)$ for an appropriate subset $A$ of $E$. Similarly we can also attain $\text{Cap}^\s(A)\leq \text{Cap}^\mu(A)$. The conclusion then follows from \cite[Theorem~1.3.4]{CF12}.
\end{proof}
\begin{remark}
Denote the resurrected Dirichlet form of $(\EE,\FF)$ by $(\EE^\mathrm{res},\FF^\mathrm{res})$. In the light of \cite[Theorems~5.1.4 and 5.2.17]{CF12}, $(\EE,\FF)$, $(\EE^\mathrm{res},\FF^\mathrm{res})$, $(\EE^\mu,\FF^\mu)$ and $(\EE^\mathrm{s}, \FF^\mathrm{s})$ all share the same set of quasi-notions. 
\end{remark}

Furthermore, we can also characterize the extended Dirichlet space of \eqref{EQ3FSU} and the global properties of snapping out Markov process. 

\begin{proposition}\label{PRO38}
Let $(\EE,\FF)$ and $(\EE^\mathrm{s},\FF^\mathrm{s})$ be in Corollary~\ref{COR36}. Then the extended Dirichlet space of $(\EE^\mathrm{s},\FF^\mathrm{s})$ is given by
\begin{equation}\label{EQ3FSEU}
	\FF^\mathrm{s}_\mathrm{e}=\left\{u\in \FF_\mathrm{e}: \int_{E\times E}\left(u(x)-u(y)\right)^2\mu(dx)\mu(dy)<\infty\right\},
\end{equation}
where $\FF_\mathrm{e}$ is the extended Dirichlet space of $(\EE,\FF)$. Particularly, 
the following assertions hold:
\begin{itemize}
\item[(1)] $(\EE,\FF)$ is recurrent, if and only if $(\EE^\mathrm{s},\FF^\mathrm{s})$ is recurrent. 
\item[(2)] If $(\EE,\FF)$ is transient, then $(\EE^\mathrm{s},\FF^\mathrm{s})$ is transient. If $(\EE,\FF)$ is irreducible, then the transience of $(\EE^\mathrm{s},\FF^\mathrm{s})$ also implies the transience of $(\EE,\FF)$. 
\item[(3)] If $(\EE,\FF)$ is recurrent or local, then the irreducibility of $(\EE,\FF)$ implies the irreducibility of $(\EE^\mathrm{s},\FF^\mathrm{s})$.              
\end{itemize}
\end{proposition}
\begin{proof}
Since $\FF^\mathrm{s}\subset \FF$ and $\EE^\mathrm{s}(u,u)\geq \EE(u,u)$ for any $u\in \FF^\mathrm{s}$, it follows from the definition of extended Dirichlet space that $\FF^\mathrm{s}_\mathrm{e}\subset \FF_\mathrm{e}$. Thus $\FF^\mathrm{s}_\mathrm{e}$ is a subset of the right side of \eqref{EQ3FSEU}. A first step towards to the contrary is to assume $u\in\FF_\mathrm{e}$ is bounded. By \cite[Theorem~2.3.4]{CF12}, we can take an approximation sequence $\{u_n\}\subset \FF$ of uniformly bounded $\EE$-quasi continuous functions for $u$. Without loss of generality, we may assume $\{u_n\}$ is $\EE$-Cauchy and $u_n$ converges to $u$, $\EE$-q.e. Since
\[
\begin{aligned}
	\lim_{n\rightarrow\infty}\int_{E\times E} &\left((u-u_n)(x)- (u-u_n)(y)\right)^2\mu(dx)\mu(dy) \\
	&=\int_{E\times E}\lim_{n\rightarrow\infty} \left((u-u_n)(x)- (u-u_n)(y)\right)^2\mu(dx)\mu(dy)=0
\end{aligned}
\]
by the bounded convergence theorem, we can deduce that $\{u_n\}$ is also $\EE^\mathrm{s}$-Cauchy.  Thus $u\in \FF^\mathrm{s}_\mathrm{e}$. Now take an arbitrary function $v$ in the right side of \eqref{EQ3FSEU}. For any $l\in \mathbb{N}$, set $v_l:=(-l)\vee v\wedge l$. Then $v_l\in \FF^\mathrm{s}_\mathrm{e}$. On the other hand,
\[
\begin{aligned}
\EE^\mathrm{s}(v_l,v_l)&=\EE(v_l,v_l)+\int_{E\times E}\left(v_l(x)-v_l(y)\right)^2\mu(dx)\mu(dy)\\
&\leq \EE(v,v)+\int_{E\times E}\left(v(x)-v(y)\right)^2\mu(dx)\mu(dy) \\
&<\infty.
\end{aligned}\] 
This implies $\sup_{l}\EE^\mathrm{s}(v_l,v_l)<\infty$. By \cite[Theorem~1.1.12]{FOT11}, we can conclude $v\in \FF^\mathrm{s}_\mathrm{e}$. 

The first and second assertions about the global properties of $(\EE^\mathrm{s},\FF^\mathrm{s})$ can be directly deduced from \cite[Theorem~2.1.8]{CF12} and \cite[Theorem~2.1.9]{CF12}. The final assertion is implied by \cite[Theorem~5.2.16]{CF12} and \cite[Theorem~4.6.4]{FOT11}.
\end{proof}
\begin{remark}
If $(\EE,\FF)$ is not irreducible, then the transience of $(\EE^\mathrm{s},\FF^\mathrm{s})$ is not sufficient for that of $(\EE,\FF)$, see Example~\ref{EXA313}.
The converse of third assertion does not always hold either, see Proposition~\ref{PRO310}. 
\end{remark}

\subsection{SNOB from approach of Dirichlet forms}

Let us turn to the snapping out Brownian motion by means of Dirichlet forms. Let $E=\mathbb{G}$ and $m$ be the Lebesgue measure on $\mathbb{G}$, i.e. $m_-:=m|_{\mathbb{G}_-}$ and $m_+:=m|_{\mathbb{G}_+}$ are the Lebesgue measures on $\mathbb{G}_-$ and $\mathbb{G}_+$ respectively. As mentioned in Remark~\ref{RM33}, the two-sided reflecting Brownian motion $(R_t)_{t\geq 0}$ on $\mathbb{G}$ is $m$-symmetric and clearly its Dirichlet form is
\begin{equation}\label{EQ3FUL}
\begin{aligned}
 & \FF=\left\{u\in L^2(\mathbb{G},m): 
 u_+\in H^1\left(\mathbb{G}_+\right), ~u_-\in H^1\left(\mathbb{G}_-\right)\right\},\\
 & \EE(u,v)=\frac{1}{2}\int_{-\infty}^{0-}u'(x)v'(x)dx+\frac{1}{2}\int_{0+}^{\infty}u'(x)v'(x)dx,\quad u,v\in \FF,
 \end{aligned}
\end{equation}
where $u_+:=u|_{\mathbb{G}_+}$, $u_-:=u|_{\mathbb{G}_-}$ and $H^1$ denotes the 1-Sobolev space, i.e.
\[
\begin{aligned}
	&H^1\left(\mathbb{G}_-\right)  :=\{u\in L^2\left(\mathbb{G}_-\right): 
	 u\text{ is absolutely continuous on }\mathbb{G}_- \text{ and }u'\in  L^2\left(\mathbb{G}_-\right) \}, \\
	& H^1\left(\mathbb{G}_+\right)  :=\{u\in L^2\left(\mathbb{G}_+\right): 
		 u\text{ is absolutely continuous on }\mathbb{G}_+ \text{ and }u'\in  L^2\left(\mathbb{G}_+\right) \}.
\end{aligned}
\]
Though every function in $H^1\left(\mathbb{G}_-\right)$ (resp. $H^1\left(\mathbb{G}_+\right)$) is well defined at $0-$ (resp. $0+$), $u\in \FF$ is not necessarily such that $u(0+)=u(0-)$. For $u,v\in \FF$, write
\[
	\int_{\mathbb{G}} u'(x)v'(x)dx:=\int_{-\infty}^{0-}u'(x)v'(x)dx+\int_{0+}^{\infty}u'(x)v'(x)dx
\] 
for convenience. The following proposition contributes to the understanding of SNOB. 

\begin{proposition}\label{PRO310}
Let $R$ be the two-sided reflecting Brownian motion on $\mathbb{G}$ associated with the Dirichlet form \eqref{EQ3FUL}. Then the SNOB $Y=(Y_t)_{t\geq 0}$ is the snapping out Markov process with respect to $R$ and $\mu:=\frac{\kappa}{2}\left(\delta_{\{0+\}}+\delta_{\{0-\}}\right)$ with $\kappa$ being the parameter in Definition~\ref{DEF31}. Furthermore, the following assertions hold:
\begin{itemize}
\item[(1)] The Dirichlet form of SNOB on $L^2(\mathbb{G},m)$ is regular and given by
\begin{equation}\label{EQ3FSUL}
\begin{aligned}
 & \FF^\mathrm{s}=\left\{u\in L^2(\mathbb{G},m): 
 u_+\in H^1\left(\mathbb{G}_+\right), ~u_-\in H^1\left(\mathbb{G}_-\right)\right\}, \\
 &\EE^\mathrm{s}(u,v)=
 \frac{1}{2}\int_\mathbb{G} u'(x)v'(x)dx+\dfrac{\kappa}{4}(u(0+)-u(0-))(v(0+)-v(0-)),\quad u,v\in \FF^\mathrm{s}.
 \end{aligned}
\end{equation}
\item[(2)] The extended Dirichlet space $\FF^\mathrm{s}_\mathrm{e}$ of $(\EE^\mathrm{s},\FF^\mathrm{s})$ is identified with that of $(\EE,\FF)$, i.e. 
\begin{equation}\label{EQ3FSE}
\begin{aligned}
\FF^\mathrm{s}_\mathrm{e}=\{u: 
u_+,& u_-\mbox{ are absolutely continuous on}~\mathbb{G}_+\\
&\qquad \qquad \qquad \mbox{and}~\mathbb{G}_- ~\mbox{respectively}, ~\EE(u,u)<\infty\}.
\end{aligned}
\end{equation}
\item[(3)] $(\EE^\mathrm{s},\FF^\mathrm{s})$ is irreducible and recurrent. Particularly, for any $x,y\in \mathbb{G}$, 
\begin{equation}\label{EQ3PXS}
	\mathbf{P}_x(\sigma_y<\infty)=1,
\end{equation}
where $\sigma_y:=\inf\{t>0: Y_t=y\}$ is the hitting time of $\{y\}$ relative to the SNOB. 
\item[(4)] The $\sigma$-finite symmetric measure of $Y$ is unique up to a constant, in other words, if another non-trivial $\sigma$-finite measure $\tilde{m}$ on $\mathbb{G}$ is  such that $Y$ is also $\tilde{m}$-symmetric, then $\tilde{m}=c\cdot m$ for some constant $c>0$.
\item[(5)] Let $f\in L^1(\mathbb{G},m)$ be Borel measurable. Then it holds $\mathbb{P}_x$-a.s. for any $x\in \mathbb{G}$ that
\[
	\lim_{t\rightarrow \infty}\frac{1}{t}\int_0^t f(Y_u)du=0. 
\]  
\end{itemize}
\end{proposition}
\begin{proof}
The first and second assertions can be deduced directly from Theorem~\ref{THM35} and  Proposition~\ref{PRO38}. The third assertion is implied by \eqref{EQ3FSE}. In fact, it follows from \eqref{EQ3FSE} that $1\in \FF^\mathrm{s}_\mathrm{e}$ and $\EE^\mathrm{s}(1,1)=0$. Then \cite[Theorem~2.1.8]{CF12} indicates the recurrence of $(\EE^\mathrm{s},\FF^\mathrm{s})$. In addition, $\EE^\mathrm{s}(u,u)=0$ with $u\in \FF^\mathrm{s}_\mathrm{e}$ clearly implies that $u$ is constant. Then from \cite[Theorem~5.2.16]{CF12} we can obtain the irreducibility of $(\EE^\mathrm{s},\FF^\mathrm{s})$. Note that the $\mathcal{E}$-polar set has to be empty and so does the $\EE^\mathrm{s}$-polar set by Corollary~\ref{COR36}. Then \eqref{EQ3PXS} can be concluded by \cite[Theorem~4.7.1]{FOT11}. For the uniqueness of symmetric measure, it suffices to note that $Y$ is finely irreducible in the sense of \cite{YZ10} and the fourth assertion holds by \cite[Theorem~2.1]{YZ10}. The final assertion is a consequence of \cite[Theorem~4.7.3]{FOT11}. 
\end{proof}
\begin{remark}
It is worth noting that the two-sided reflecting Brownian motion $R$ on $\mathbb{G}$ is not irreducible and its symmetric measures are not unique. In fact, all the non-trivial symmetric measures of $R$ can be written as
\[
	\left\{c_1dx|_{\mathbb{G}_-}+c_2dx|_{\mathbb{G}_+}: c_1,c_2>0\right\}. 	
\]
Intuitively speaking, the snapping out method builds a `bridge' between $0+$ and $0-$ and links the two separate components of $R$, so that the SNOB becomes irreducible and its symmetric measure is unique. 
\end{remark}

We complete this subsection with an interesting link between SNOB and one-dimensional Brownian motion. For any $\beta>0$,
\[
	T_{\beta}: \mathbb{G}\rightarrow (-\infty, -\beta]\cup [\beta, \infty),\quad x\mapsto \left\lbrace
	\begin{aligned}
	&x+\beta,\quad x\in \mathbb{G}_+, \\
	&x-\beta,\quad x\in \mathbb{G}_-
	\end{aligned} \right. 
\]
denotes the homeomorphism between $\mathbb{G}$ and $(-\infty, -\beta]\cup [\beta, \infty)$. Note that $T_\beta(Y):=(T_\beta(Y_t))_{t\geq 0}$ is a Markov process on $(-\infty, -\beta]\cup [\beta, \infty)$. The following result tells us the darning of SNOB by shorting $\{0+,0-\}$ into $0$ is the one-dimensional Brownian motion, and on the contrary, the SNOB is the trace of one-dimensional Brownian motion up to a spatial transform. 

\begin{theorem}\label{THM316}
\begin{itemize}
\item[(1)] Let $Y$ be the SNOB on $\mathbb{G}$ associated with the Dirichlet form \eqref{EQ3FSUL}. By shorting $\{0+,0-\}$ into $0$, the Markov process with darning induced by $Y$ (Cf. \S\ref{D}) is nothing but the one-dimensional Brownian motion. 
\item[(2)] Let $(\frac{1}{2}\mathbf{D}, H^1(\mathbb{R}))$ be the associated Dirichlet form  of one-dimensional Brownian motion on $L^2(\mathbb{R})$. Set $F_\kappa:=(-\infty, -\kappa^{-1}]\cup [\kappa^{-1}, \infty)$ and $m_\kappa:= m|_{F_\kappa}$ with $m$ being the Lebesgue measure on $\mathbb{R}$. Then $T_{\kappa^{-1}}(Y)$ is a Markov process on $F_\kappa$ associated with the trace Dirichlet form of $(\frac{1}{2}\mathbf{D}, H^1(\mathbb{R}))$ on $F_\kappa$ with the speed measure $m_\kappa$. 
\end{itemize}
\end{theorem}
\begin{proof}
The first assertion is clear by applying \eqref{DR}. For the second assertion, let $(\check{\EE},\check{\FF})$ be the trace Dirichlet form of $(\frac{1}{2}\mathbf{D}, H^1(\mathbb{R}))$ on $F_\kappa$ with the speed measure $m_\kappa$. Clearly, 
\[
	\check{\FF}=\{f\in L^2(F_\kappa, m_\kappa): f|_{[\kappa^{-1}, \infty)}\in H^1([\kappa^{-1}, \infty)), f|_{(-\infty, -\kappa^{-1}]}\in H^1((-\infty, -\kappa^{-1}])\}. 
\]
Following the proof of \cite[Theorem~2.1]{LY17}, we can deduce that for any $f\in \check{\FF}$, 
\[
\begin{aligned}
	\check{\EE}(f,f)&=\frac{1}{2}\int_{F_\kappa} f'(x)^2dx +\frac{1}{2}\frac{\left(f(\kappa^{-1})-f(-\kappa^{-1})\right)^2}{|\kappa^{-1} - (-\kappa^{-1})|} \\
	&=\frac{1}{2}\int_{F_\kappa} f'(x)^2dx +\frac{\kappa}{4}\left(f(\kappa^{-1})-f(-\kappa^{-1})\right)^2.
\end{aligned}\]
Clearly, $T_{\kappa^{-1}}(Y)$ is associated with $(\check{\EE},\check{\FF})$. That completes the proof. 
\end{proof}

\subsection{Snapping out diffusion processes on $\mathbb{G}$}\label{SEC35}

We present a family of more general snapping out Markov processes on $\mathbb{G}$, which will be used in \S\ref{SEC4}. The symmetric measure (not necessarily the Lebesgue measure) is still denoted by $m$. Let $\mathscr{M}$ be the family of fully supported positive Radon measures on $\mathbb{G}$ charging no set of singleton. In other words, 
\begin{equation}\label{EQ3MNA}
\mathscr{M}:=\left\{\nu: \text{ a fully supported Radon measure on }\mathbb{G}\text{ and } \nu\left(\{x\}\right)=0,\forall x\in \mathbb{G}\right\}.
\end{equation}
Then $\nu \in \mathscr{M}$ indicates $0<\nu_+([a, b]), \nu_-([-b,-a])<\infty$ for $0\leq a<b$, where $\nu_\pm:=\nu|_{\mathbb{G}_\pm}$. Clearly, every $\nu\in \mathscr{M}$ induces a fully supported Radon measure on $\mathbb{R}$ charging no set of singleton. We should use the same symbol $\nu$ for it if no confusion caused.

Fix $\lambda\in \mathscr{M}$ and denote $\lambda_\pm:=\lambda|_{\mathbb{G}_\pm}$ as usual. Clearly, $\lambda_\pm$ induces a unique scale function $\ss_\pm$ on $\mathbb{G}_\pm$ such that $\ss_\pm(0\pm)=0$, in other words, 
\[
	\ss_+(x)=\lambda_+([0+,x]),\quad \ss_-(-x)=-\lambda_-([-x,0-]),\quad x\in \mathbb{G}_+. 
\] 
Denote the combination of $\ss_\pm$ by $\ss$, i.e. $\ss(x):=\ss_+(x)$ for $x\geq 0$ and $\ss(x):=\ss_-(x)$ for $x<0$. Then $\ss$ is the scale function on $\mathbb{R}$ induced by $\lambda$.

A first step towards the snapping out diffusion processes on $\mathbb{G}$ is to start with a diffusion $X$ on $\mathbb{G}$ as a union of $X^+$ and $X^-$, where $X^\pm$ is an irreducible diffusion on $\mathbb{G}_\pm$ with scale function $\ss_\pm$, speed measure $m_\pm$ and no killing inside. In other words, $X^\pm$ is given by the regular Dirichlet form on $L^2(\mathbb{G}_\pm, m_\pm)$:  (see \cite{LY172})
\[
\begin{aligned}
&\FF^\pm=\bigg\{f\in L^2(\mathbb{G}_\pm, m_\pm): f\ll \lambda_\pm, \int_{\mathbb{G}_\pm} \left(\frac{df}{d\lambda_\pm}\right)^2d\lambda_\pm<\infty, \\
&\qquad\qquad\qquad\qquad \qquad f(\pm\infty):=\lim_{x\rightarrow \pm\infty}f(x)=0 \text{ if }\lambda_\pm(\mathbb{G}_\pm)<\infty\bigg\}, \\
&\EE^\pm(f,g)=\frac{1}{2}\int_{\mathbb{G}_\pm}\frac{df}{d\lambda_\pm}\frac{dg}{d\lambda_\pm}d\lambda_\pm,\quad f,g\in \FF^\pm, 
\end{aligned}
\]
and $X$ is associated with the regular Dirichlet form on $L^2(\mathbb{G},m)$
\begin{equation}\label{EQ4FFLG}
\begin{aligned}
	\FF&=\{f\in L^2(\mathbb{G},m): f_+\in \FF^+, f_-\in \FF^-\}, \\
	\EE(f,g)&=\EE^+(f_+, g_+)+\EE^-(f_-,g_-),\quad f,g\in \FF, 
\end{aligned}
\end{equation}
where $f_\pm:=f|_{\mathbb{G}_\pm}$. Note that $(\EE,\FF)$ is not irreducible. Its extended Dirichlet space is
\[
\FF_\e=\{f: f_\pm \in \FF^\pm_\e\}, 
\]
where (Cf. \cite[Theorem~2.2.11 and (3.5.11)]{CF12})
\[
\begin{aligned}	
&\FF^\pm_\e=\bigg\{f: f\ll \lambda_\pm, \int_{\mathbb{G}_\pm} \left(\frac{df}{d\lambda_\pm}\right)^2d\lambda_\pm<\infty, f(\pm\infty)=0 \text{ if }\lambda_\pm(\mathbb{G}_\pm)<\infty\bigg\}, \\
\end{aligned}
\]

The snapping out diffusion process $X^\s$ is, by definition, the snapping out Markov process with respect to $X$ and a finite smooth measure $\mu$. The smooth measures we are interested in are those supported on $\{0+,0-\}$, in other words, 
\begin{equation}\label{EQ3MKD}
	\mu=\kappa_+\cdot \delta_{\{0+\}}+\kappa_-\cdot \delta_{\{0-\}}
\end{equation}
for some constants $\kappa_\pm>0$. By applying Theorem~\ref{THM35}, we can conclude the following result. 

\begin{proposition}\label{PRO313}
Let $\lambda\in \mathscr{M}$ and $\mu$ be in \eqref{EQ3MKD}. Then $X^\s$ is $m$-symmetric on $\mathbb{G}$ and associated with a regular Dirichlet form on $L^2(\mathbb{G},m)$
\[
\begin{aligned}
	\FF^\s &= \FF, \\
	\EE^\s(f,g)&=\EE(f,g)+\frac{\kappa_+\kappa_-}{\kappa_++\kappa_-}\left(f(0+)-f(0-)\right)\left(g(0+)-g(0-)\right),\quad f,g\in \FF. 
\end{aligned}
\] 
Its extended Dirichlet space is $\FF^\s_\e=\FF_\e$. Furthermore, the following hold: 
\begin{itemize}
\item[(1)] $(\EE^\s,\FF^\s)$ is irreducible and particularly, for any $x,y\in \mathbb{G}$,
\begin{equation}\label{EQ3PXY}
	\mathbf{P}_x(\sigma_y<\infty)>0,
\end{equation}
where $\sigma_y$ is the hitting time of $\{y\}$ relative to $X^\s$.    
\item[(2)] $(\EE^\s,\FF^\s)$ is transient, if and only if either $\lambda_+(\mathbb{G}_+)<\infty$ or $\lambda_-(\mathbb{G}_-)<\infty$. Otherwise, it is recurrent.   
\item[(3)] The $\sigma$-finite symmetric measure of $X^\s$ is unique up to a constant. 
\end{itemize}
\end{proposition}
\begin{proof}
Note that for any $f\in \FF_\e$, $f(0\pm)$ exists and is finite. Thus $\FF^\s=\FF$ and $\FF^\s_\e=\FF_\e$ by \eqref{EQ3FSU} and \eqref{EQ3FSEU}. 

Let us show the irreducibility of $(\EE^\s,\FF^\s)$. Then \eqref{EQ3PXY} is implied by the fact that every singleton is of positive capacity relative to $\EE^\s$ obtained by Corollary~\ref{COR36}. Suppose $A$ is an invariant set (Cf. \cite[\S2.1]{CF12}) of $(\EE^\s,\FF^\s)$. Then by \cite[Proposition~2.1.6]{CF12}, we may easily deduce that $A\cap \mathbb{G}_\pm$ is an invariant set of $(\EE^\pm, \FF^\pm)$. Since $(\EE^\pm, \FF^\pm)$ is irreducible, it follows that $A=\emptyset, \mathbb{G}_+, \mathbb{G}_-$ or $\mathbb{G}$. Suppose $A=\mathbb{G}_+$. By using \cite[Proposition~2.1.6]{CF12} again, we have 
\[
	\EE^\s(f,g)=\EE^\s(f_+, g_+)+\EE^\s(f_-, g_-)
\]
for any $f,g\in \FF^\s$. However, the right-hand side is equal to
\[
\EE^+(f_+,g_+)+\EE^-(f_-,g_-)+\frac{\kappa_+\kappa_-}{\kappa_++\kappa_-}\left(f(0+)g(0+)+f(0-)g(0-)\right)\neq \EE^\s(f,g)
\]
for $f,g$ satisfy $f(0-)g(0+)+f(0+)g(0-)\neq 0$. This leads to $A\neq \mathbb{G}_+$. Similarly, we can obtain $A\neq \mathbb{G}_-$ and therefore, $A=\emptyset$ or $\mathbb{G}$. 

Next, we prove the second assertion. For the sufficiency of transience, there is no loss of generality in assuming $\lambda_+(\mathbb{G}_+)<\infty$. Suppose $f\in \FF^\s_\e$ with $\EE^\s(f,f)=0$. This implies $f_\pm \in \FF^\pm_\e$, $\EE^\pm(f_\pm, f_\pm)=0$ and $f(0+)=f(0-)$. It follows from $\lambda_+(\mathbb{G}_+)<\infty$ that $(\EE^+,\FF^+)$ is transient. Hence $f_+=0$. Moreover, $\EE^-(f_-,f_-)=0$ indicates $f_-$ is constant on $\mathbb{G}_-$. Then $f(0+)=f(0-)$ tells us $f=0$ on $\mathbb{G}$. To the contrary, we need only note if $\lambda_\pm(\mathbb{G}_\pm)=\infty$, then $(\EE^\pm,\FF^\pm)$ is recurrent by \cite[Theorem~2.2.11]{CF12} and thus $(\EE^\s,\FF^\s)$ is also recurrent by Proposition~\ref{PRO38}. 

The final assertion can be obtained by mimicking the proof of Proposition~\ref{PRO310}. That completes the proof. 
\end{proof}
\begin{remark}
By shoring $\{0+,0-\}$ into $0$, the darning transform on $X^\s$ leads to an irreducible diffusion on $\mathbb{R}$ with scale function $\ss$, speed measure $m$ and no killing inside.  
\end{remark}

We complete this subsection with several concrete examples. The first example sheds light on the significance of symmetric measure $m$ in the snapping out method. 

\begin{example}\label{EXA311}
Let us consider the two-sided reflecting Brownian motion $R$ on $\mathbb{G}$ but take a different symmetric measure $\tilde{m}(dx):=2(1-\alpha)\cdot dx|_{\mathbb{G}_-}+2\alpha\cdot dx|_{\mathbb{G}_+}$ with a constant $0<\alpha<1$. Its Dirichlet form on $L^2(\mathbb{G},\tilde{m})$ is written as
\begin{equation}\label{EQ3FULG}
\begin{aligned}
 & \FF=\left\{u\in L^2(\mathbb{G},\tilde{m}): 
 u_+\in H^1\left(\mathbb{G}_+\right), ~u_-\in H^1\left(\mathbb{G}_-\right)\right\},\\
 & \EE(u,v)=(1-\alpha)\int_{-\infty}^{0-}u'(x)v'(x)dx+\alpha\int_{0+}^{\infty}u'(x)v'(x)dx,\quad u,v\in \FF.
 \end{aligned}
\end{equation}
Let $\mu=\frac{\kappa}{2}\left(\delta_{\{0+\}}+\delta_{\{0-\}}\right)$. Note that the killing transforms of \eqref{EQ3FULG} and \eqref{EQ3FUL} induced by the same measure $\mu$ are different, since the PCAFs of $\mu$  are different with respect to different symmetric measures. 

The snapping out Markov process $\tilde{Y}$ with respect to \eqref{EQ3FULG} and $\mu$ is also $\tilde{m}$-symmetric and its associated regular Dirichlet form on $L^2(\mathbb{G},\tilde{m})$ is
\[
\begin{aligned}
  \FF^\mathrm{s}&=\FF, \\
 \EE^\mathrm{s}(u,v)&=
 \EE(u,v)+\dfrac{\kappa}{4}(u(0+)-u(0-))(v(0+)-v(0-)),\quad u,v\in \FF^\mathrm{s}.
 \end{aligned}
\]
It is is irreducible and recurrent by Proposition~\ref{PRO313}. The symmetric measure of $\tilde{Y}$ is unique up to a constant. Particularly, if $\alpha\neq 1/2$, then $\tilde{Y}$ is not symmetric with respect to the Lebesgue measure on $\mathbb{G}$. 
\end{example}

The next example gives the so-called \emph{$\alpha$-skew SNOB}. 

\begin{example}
Let $(\EE,\FF)$ be the regular Dirichlet form \eqref{EQ3FULG} of $R$ on $L^2(\mathbb{G},\tilde{m})$. Take another smooth measure $\mu_\alpha:=(1-\alpha)\kappa \delta_{\{0-\}}+\alpha\kappa \delta_{\{0+\}}$. The Dirichlet form of snapping out Markov process with respect to \eqref{EQ3FULG} and $\mu_\alpha$ is 
\begin{equation}\label{EQ3FSFE}
\begin{aligned}
  \FF^\mathrm{s}&=\FF, \\
 \EE^\mathrm{s}(u,v)&=
 \EE(u,v)+\alpha(1-\alpha)\kappa(u(0+)-u(0-))(v(0+)-v(0-)),\quad u,v\in \FF^\mathrm{s}.
 \end{aligned}
\end{equation}
We call this snapping out Markov process the $\alpha$-skew SNOB and denote it by $Y^\alpha $. This name follows the so-called $\alpha$-skew Brownian motion in \cite{HS81}. Indeed, after shorting $\{0+,0-\}$ into $0$ and applying the darning transform to \eqref{EQ3FSFE}, we can obtain the associated Dirichlet form of $\alpha$-skew Brownian motion.  Particularly, when $\alpha=1/2$, the $\alpha$-skew SNOB is nothing but the SNOB. 

Mimicking \cite[Proposition~1]{L16}, we can deduce that $Y^\alpha$ is related to the heat equation \eqref{EQ3TVT} and the condition of discontinuous flux at $0$: 
\[
\begin{aligned}
&\alpha \nabla u(t,0+)=(1-\alpha)\nabla u(t,0-),\\
&(1-\alpha)\kappa (u(t,0+)-u(t,0-))=\nabla u(t,0+).
\end{aligned}
\]
See \S\ref{SEC44} for more considerations about this boundary condition. 
\end{example}

Another example below shows that the transience of $(\EE^\mathrm{s},\FF^\mathrm{s})$ is not sufficient for that of $(\EE,\FF)$ if $(\EE,\FF)$ is not irreducible.

\begin{example}\label{EXA313}
In Proposition~\ref{PRO313}, take $\ss_-(x)=x$ and $\ss_+(x):=1-\mathrm{e}^{-x}$. Then $X^-$ is recurrent, while $X^+$ is transient by \cite[Theorem~2.2.11]{CF12}. Thus $(\EE,\FF)$ is neither transient nor recurrent. However, since $\lambda_+(\mathbb{G}_+)<\infty$, we know that $(\EE^\s,\FF^\s)$ is transient by Proposition~\ref{PRO313}. 
\end{example}

\subsection{Other examples}\label{SEC34}

Two more examples of snapping out Markov processes are presented below. The first one is based on a diffusion on $\mathbb{R}$, which consists of a countable set of separate reflecting Brownian motions. 

\begin{example}
Let $K$ be the standard Cantor set and write $K^c$ as a union of disjoint open intervals:
\[
	K^c=\cup_{n\geq 1} (a_n,b_n),
\]
where $(a_1,b_1)=(1,\infty)$ and $(a_2,b_2)=(-\infty,0)$. 
We use the conventions $[a_1,b_1]:=[1,\infty)$ and $[a_2,b_2]=(-\infty,0]$ for convenience. For each $n\geq 1$, denote the associated Dirichlet form on $L^2([a_n,b_n])$ of reflecting Brownian motion on $[a_n,b_n]$ by $(\EE^n,\FF^n)$. Set
\[
\begin{aligned}
	\FF&=\{u\in L^2(\mathbb{R}):u|_{[a_n,b_n]}\in \FF^n, n\geq 1\}, \\
	\EE(u,v)&=\sum_{n\geq 1}\EE^n(u|_{[a_n,b_n]}, v|_{[a_n,b_n]}),\quad u,v\in \FF.
\end{aligned}
\]
Then $(\EE,\FF)$ is a regular Dirichlet form on $L^2(\mathbb{R})$ due to \cite{LY172}. Note that $\mathbb{R}\setminus \cup_{n\geq 1}[a_n,b_n]$ is $\EE$-polar and $\{x\}$ is of positive capacity for any $x\in \cup_{n\geq 1}[a_n,b_n]$. Roughly speaking, the associated Markov process of $(\EE,\FF)$ is a disjoint union of countable reflecting Brownian motions.

Let $\mu$ be a smooth probability measure on $\mathbb{R}$, in other words, $\mu\left(\mathbb{R}\setminus \cup_{n\geq 1}[a_n,b_n]\right)=0$. Assume that
\[
	\mu_n:=\mu([a_n,b_n])>0,\quad \forall n\geq 1. 
\]
For example, 
\[
	\mu= \frac{1}{4}\left(\delta_{\{0\}}+\delta_{\{1\}}\right)+\sum_{n\geq 3}\frac{1}{2^n} \left(\delta_{\{a_n\}}+\delta_{\{b_n\}}\right).
\]
Then the snapping out Markov process with respect to $(\EE,\FF)$ and $\mu$ is irreducible and recurrent. This fact can be attained by mimicking the proof of Proposition~\ref{PRO310} and we omit its details. Intuitively speaking, if $\mu_n, \mu_m>0$, then the snapping out method builds a `bridge' between $[a_n,b_n]$ and $[a_m,b_m]$ by additional jumps.
\end{example}

The next example starts with a pure-jump process on $\mathbb{G}$. 

\begin{example}
Consider a regular Dirichlet form $(\mathcal{B},\mathcal{W})$ on $L^2(\mathbb{G}_+)=L^2([0+,\infty))$ for $1<\alpha<2$: 
\[
\begin{aligned}
 \mathcal{W} &=\left\{u\in L^2(\mathbb{G}_+): |u|<\infty~\mbox{a.e.}, \mathcal{B}(u,u)<\infty \right\},\\
  \mathcal{B}(u,v)&=c\int_{\mathbb{G}_+\times \mathbb{G}_+\setminus d_+}\dfrac{(u(x)-u(y))(v(x)-v(y))}{|x-y|^{1+\alpha}}dxdy,\quad  u,v\in \mathcal{W},
 \end{aligned}
\]
where $d_+$ is the diagonal of $\mathbb{G}_+\times \mathbb{G}_+$ and $c>0$ is a constant depending on $\alpha$ (see \cite{BBC03}). The associated process is called the \emph{reflecting $\alpha$-stable process} on $\mathbb{G}_+$. It is irreducible and recurrent, and every singleton is of positive capacity. We refer to \cite{BBC03} for more details about these facts. Mimicking the two-sided reflecting Brownian motion on $\mathbb{G}$, we extend the reflecting $\alpha$-stable process to a two-sided one $X=(X_t)_{t\geq 0}$ on $\mathbb{G}=\mathbb{G}_+\cup \mathbb{G}_-$ by symmetry. Namely, $X$ is given by the regular Dirichlet form on $L^2(\mathbb{G},m)$ ($m$ is the Lebesgue measure on $\mathbb{G}$) as follows: 
\[
\begin{aligned}
  \FF&=\left\{u\in L^2(\mathbb{G},m): |u|<\infty~\mbox{a.e.}, \EE(u,u)<\infty \right\},\\
 \EE(u,v)&=c\int_{\left(\mathbb{G}_+\times \mathbb{G}_+\right)\cup \left(\mathbb{G}_-\times \mathbb{G}_-\right)\setminus d}\dfrac{(u(x)-u(y))(v(x)-v(y))}{|x-y|^{1+\alpha}}dxdy,\quad u,v\in \FF,
 \end{aligned}
\]
where $d$ is the diagonal of $\left(\mathbb{G}_+\times \mathbb{G}_+\right)\cup \left(\mathbb{G}_-\times \mathbb{G}_-\right)$. 
Clearly, $(\EE,\FF)$ is recurrent but not irreducible.  

Take $\mu=\frac{1}{2}(\delta_{\{0+\}}+\delta_{\{0-\}})$, which is a smooth probability measure with respect to $(\EE,\FF)$. The snapping out Markov process with respect to $X$ and $\mu$ is denoted by $X^\s$ and we call it the \emph{snapping out $\alpha$-stable process}. Its associated Dirichlet form is
\[
\begin{aligned}
 \FF^\s&=\left\{u\in L^2(\mathbb{G},m): |u|<\infty~\mbox{a.e.}, \EE^\s(u,u)<\infty \right\},\\
 \EE^\s(u,v)&=c\int_{\left(\mathbb{G}_+\times \mathbb{G}_+\right)\cup \left(\mathbb{G}_-\times \mathbb{G}_-\right)\setminus d}\dfrac{(u(x)-u(y))(v(x)-v(y))}{|x-y|^{1+\alpha}}dxdy\\
 &\qquad \qquad \qquad +\dfrac{1}{4}\left(u(0+)-u(0-)\right)\left(v(0+)-v(0-)\right),\quad u,v\in \FF^\s. 
  \end{aligned}
\]
Clearly, $X^\s$ is also a pure-jump process and mimicking the proof of Proposition~\ref{PRO310}, we can conclude that $X^\s$ is irreducible and recurrent. 
\end{example}

\section{Stiff problems in one-dimensional space}\label{SEC4}

This section is devoted to explore the stiff problem in $\mathbb{R}$ via Dirichlet forms. We shall first introduce the Mosco convergence of Dirichlet forms. It will be used in \S\ref{SEC43} to build a phase transition of stiff problem as the length of the normal barrier decreases to zero. Then in \S\ref{SEC42} we shall give three Markov processes on $\mathbb{G}$ or $\mathbb{R}$, which are the probabilistic counterparts of thermal conductions in stiff problem. In what follows, the general stiff problem in one-dimensional space will be phrased and solved. 

\subsection{Mosco convergence of Dirichlet forms}\label{SEC41}

Mosco convergence raised in \cite{U94} is a kind of convergence for closed forms. We shall write down its specific definition for handy reference. Let $(\EE^n,\FF^n)$ be a sequence of closed forms on a same Hilbert space $L^2(E,m)$, and $(\EE,\FF)$ be another closed form on $L^2(E,m)$. We always extend the domains of $\EE$ and $\EE_n$ to $L^2(E,m)$ by letting
\[
\begin{aligned}
	\EE(u,u)&:=\infty, \quad u\in L^2(E,m)\setminus \FF, \\ 
	\EE^n(u,u)&:=\infty,\quad u\in L^2(E,m)\setminus \FF^n.
\end{aligned}
\]
In other words, $u\in \FF$ (resp. $u\in \FF^n$) if and only if $\EE(u,u)<\infty$ (resp. $\EE^n(u,u)<\infty$). 
Furthermore, we say $u_n$ converges to $u$ weakly in $L^2(E,m)$, if for any $v\in L^2(E,m)$, $(u_n,v)_m\rightarrow (u,v)_m$ as $n\rightarrow \infty$, and strongly in $L^2(E,m)$, if $\|u_n-u\|_{L^2(E,m)}\rightarrow \infty$. 

\begin{definition}\label{DEF41}
Let $(\EE^n,\FF^n)$ and $(\EE,\FF)$ be given above. Then $(\EE^n,\FF^n)$ is said to be convergent to $(\EE,\FF)$ in the sense of Mosco, if
\begin{itemize}
\item[(1)] For any sequence $\{u_n:n\geq 1\}\subset L^2(E,m)$ that converges weakly to $u$ in $L^2(E,m)$, it holds that
\[
	\EE(u,u)\leq \liminf_{n\rightarrow \infty}\EE^n(u_n,u_n). 
\]
\item[(2)] For any $u\in L^2(E,m)$, there exists a sequence $\{u_n:n\geq 1\}\subset L^2(E,m)$ that converges strongly to $u$ in $L^2(E,m)$ such that
\[
	\EE(u,u)\geq \limsup_{n\rightarrow \infty}\EE^n(u_n,u_n). 
\]
\end{itemize}
\end{definition} 

Let $(T^n_t)_{t\geq 0}$ and $(T_t)_{t\geq 0}$ be the semigroups of $(\EE^n, \FF^n)$ and $(\EE,\FF)$ respectively, and $(G^n_\alpha)_{\alpha>0}, (G_\alpha)_{\alpha>0}$ be their corresponding resolvents.  The following result is well-known (Cf. \cite{U94}).

\begin{proposition}\label{PRO42}
Let $(\EE^n, \FF^n), (\EE, \FF)$ be above. Then the following are equivalent:
\begin{itemize}
\item[(1)] $(\EE^n,\FF^n)$ converges to $(\EE,\FF)$ in the sense of Mosco;
\item[(2)] for every $t>0$ and $f\in L^2(E,m)$,  $T^n_tf$ converges to $T_tf$ strongly in $L^2(E,m)$; 
\item[(3)] for every $\alpha>0$ and $f\in L^2(E,m)$,  $G^n_\alpha f$ converges to $G_\alpha f$ strongly in $L^2(E,m)$. 
\end{itemize}
\end{proposition}

\subsection{Markov processes related to the phases of stiff problem}\label{SEC42}

Recall that $\mathbb{G}=\mathbb{G}_+\cup \mathbb{G}_-$. The family $\mathscr{M}$ of measures is given by \eqref{EQ3MNA}. 
Fix $m, \lambda\in \mathscr{M}$. Denote the scale function induced by $\lambda_\pm$ by $\ss_\pm$. Their combination, i.e. the scale function induced by $\lambda$ on $\mathbb{R}$, is denoted by $\ss$ as in \S\ref{SEC35}. 
 
The following Markov processes on $\mathbb{R}$ or $\mathbb{G}$ related to $m$ and $\lambda$ are of great interest in this section: 
\begin{itemize}
\item[(1)] a two-sided diffusion process $X$ on $\mathbb{G}$, which is a union of reflecting diffusion $X^\pm:=(X^\pm_t)_{t\geq 0}$ on $\mathbb{G}_\pm$ with scale function $\ss_\pm$, speed measure $m_\pm$ and no killing inside (Cf. \cite{IM74}),
\item[(2)] the snapping out Markov process $X^\s$ on $\mathbb{G}$ with respect to $X$ and 
\[
	\mu:=\frac{\kappa}{2}\left(\delta_{\{0+\}}+\delta_{\{0-\}}\right)
\] 
with a parameter $\kappa>0$, and 
\item[(3)] a diffusion process $X^\i=(X^\i_t)_{t\geq 0}$ on $\mathbb{R}$ with scale function $\ss$, speed measure $m$ and no killing inside. 
\end{itemize}
The diffusion $X$ is given by the Dirichlet form \eqref{EQ4FFLG}. It is not irreducible, and $\mathbb{G}_+, \mathbb{G}_-$ are its invariant sets. Applying Proposition~\ref{PRO313}, $X^\s$ is associated with
\begin{equation}\label{EQ4FSF}
\begin{aligned}
	\FF^\s&=\FF,\\
	\EE^\s(f,g)&=\EE(f,g)+\frac{\kappa}{4}\left(f(0+)-f(0-)\right)\left(g(0+)-g(0-)\right),\quad f,g\in \FF^\s. 
\end{aligned}
\end{equation}
It is irreducible. Finally, the irreducible diffusion $X^\i$ (the superscript `$\i$' stands for `irreducible') is $m$-symmetric and associated with a regular Dirichlet form on $L^2(\mathbb{R},m)$
\begin{equation}\label{EQ4FIF}
\begin{aligned}
&\FF^\i=\bigg\{f\in L^2(\mathbb{R}, m): f\ll \lambda, \int_{\mathbb{R}} \left(\frac{df}{d\lambda}\right)^2d\lambda<\infty, \\
&\qquad\qquad\qquad\qquad \qquad f(\pm\infty):=\lim_{x\rightarrow \pm\infty}f(x)=0 \text{ if }\lambda_\pm(\mathbb{G}_\pm)<\infty\bigg\}, \\
&\EE^\i(f,g)=\frac{1}{2}\int_{\mathbb{R}}\frac{df}{d\lambda}\frac{dg}{d\lambda}d\lambda,\quad f,g\in \FF^\i. 
\end{aligned}
\end{equation} 
It is worth noting that every (quasi-continuous) function $f$ in $\FF$ (or $\FF^\s$) is continuous on $\mathbb{G}_+$ and $\mathbb{G}_-$ respectively, but possibly $f(0-)\neq f(0+)$. However, every (quasi-continuous) function in $\FF^\i$ is continuous on $\mathbb{R}$, particularly it is continuous at $0$. Notice that $L^2(\mathbb{G},m)=L^2(\mathbb{R},m)$. If we regard every function in $\FF$ as an $m$-equivalence class, then $\FF^i\subsetneqq \FF=\FF^\s$.

\begin{remark}\label{RM43}
The fixed measure $m\in \mathscr{M}$ is the common symmetric measure (or speed measure) of these Markov processes. It is usually taken to be the Lebesgue measure in the thermal conduction model. The scale function $\ss$ as well as $\lambda$ plays the role of the `thermal resistance', which reflects the ability of the material to resist the flow of the heat. Let us make a brief explanation of this fact. Take $m$ to be the Lebesgue measure on $\mathbb{R}$ and assume that $\ss$ is absolutely continuous. Then for any $f,g\in \FF^\i$, 
\[
	\EE^\i(f,g)=\frac{1}{2}\int_\mathbb{R} \frac{f'(x)g'(x)}{\ss'(x)}dx. 
\]
Under a slight assumption, the generator $\mathcal{L}^\i$ of $(\EE^\i,\FF^\i)$ has $C_c^\infty(\mathbb{R})$ as its core and for any $f\in C_c^\infty(\mathbb{R})$,
\[
	\mathcal{L}^if(x)=\frac{1}{2}\nabla\left(\frac{1}{\ss'(x)}\nabla f(x)\right).  
\]
In other words, $1/\ss'$ is nothing but the thermal conductivity $a$ in \eqref{EQ3PTU}. 
\end{remark}

\begin{example}
When $m$ and $\lambda$ are both the Lebesgue measure on $\mathbb{G}$, $X$ is the two-sided reflecting Brownian motion on $\mathbb{G}$, $X^\s$ is the SNOB, and $X^\i$ is the one-dimensional Brownian motion on $\mathbb{R}$.

In Example~\ref{EXA311}, $m=\tilde{m}$, $\lambda_+(dx)=\frac{dx}{2\alpha}$ and $\lambda_-(dx)=\frac{dx}{2(1-\alpha)}$. 
In Example~\ref{EXA313}, $m$ is the Lebesgue measure on $\mathbb{G}$, $\lambda_-$ is the Lebesgue measure on $\mathbb{G}_-$ but $\lambda_+$ is a finite measure on $\mathbb{G}_+$.
\end{example}

Let $H:=L^2(\mathbb{R},m)=L^2(\mathbb{G},m)$. Denote the generators of $X, X^\s, X^\i$ on $H$ by $\mathcal{L}, \mathcal{L}^\s,  \mathcal{L}^\i$ respectively. Recall that $u\in \D(\L^\dagger)$, $f=\L^\dagger u\in H$ if and only if $u\in \FF^\dagger$ and $\EE^\dagger(u,v)=(-f,v)_H$ for any $v\in \FF^\dagger$, where $\dagger$ is vacant or stands for $\s$ or $\i$ until the end of this section, and $\mathcal{D}(\mathcal{L}^\dagger)$ is the domain of $\mathcal{L}^\dagger$. 

\begin{proposition}\label{PRO45}
Let $m,\lambda,\kappa$ and $X,X^\s,X^\i$ be given above. 
\begin{itemize}
\item[(1)] The generator of $X$ is 
\[
	\L u|_{\mathbb{G}_\pm}=\frac{1}{2}\frac{d}{dm_\pm}\left(\frac{du_\pm}{d\lambda_\pm}\right)
\]
with
\begin{equation*}
\begin{aligned}
\D(\L)=\left\{u\in \FF: \frac{du_\pm}{d\lambda_\pm}\ll m_\pm, \frac{d}{dm_\pm}\left(\frac{du_\pm}{d\lambda_\pm}\right)\in L^2(\mathbb{G}_\pm,m_\pm),  \frac{du_\pm}{d\lambda_\pm}(0\pm)=0\right\}.
\end{aligned}
\end{equation*}
\item[(2)] The generator of $X^\s$ is 
\[
	\L^\s u|_{\mathbb{G}_\pm}=\frac{1}{2}\frac{d}{dm_\pm}\left(\frac{du_\pm}{d\lambda_\pm}\right)
\]
with
\begin{equation}\label{EQ4DLS}
\begin{aligned}
\D(\L^\s)=\bigg\{u\in \FF^\s: \frac{du_\pm}{d\lambda_\pm}\ll m_\pm, &\frac{d}{dm_\pm}\left(\frac{du_\pm}{d\lambda_\pm}\right)\in L^2(\mathbb{G}_\pm,m_\pm),\\ & \frac{du_\pm}{d\lambda_\pm}(0\pm)=\frac{\kappa}{2}(u(0+)-u(0-))\bigg\}.
\end{aligned}
\end{equation}
\item[(3)] The generator of $X^\i$ is 
\[
	\L^i u=\frac{1}{2}\frac{d}{dm}\left(\frac{du}{d\lambda}\right)
\]
with
\begin{equation*}
\begin{aligned}
\D(\L^\i)=\bigg\{u\in \FF^\i: \frac{du}{d\lambda}\ll m, \frac{d}{dm}\left(\frac{du}{d\lambda}\right)\in L^2(\mathbb{R},m)\bigg\}.
\end{aligned}
\end{equation*}
\end{itemize}
\end{proposition}
\begin{proof}
Note that $\frac{du_\pm}{d\lambda_\pm}\ll m_\pm$ and $\frac{d}{dm_\pm}\left(\frac{du_\pm}{d\lambda_\pm}\right)\in L^2(\mathbb{G}_\pm,m_\pm)$ imply $\frac{du_\pm}{d\lambda_\pm}$ is continuous on $\mathbb{G}\pm$ and of bounded variaiton, since $m$ charges no set of singleton.  Particularly, $\frac{du_\pm}{d\lambda_\pm}(0\pm)$ is well defined. The expressions of $\L$ and $\L^\i$ are derived in \cite{F14}. We need only prove the second assertion. Denote the right side of \eqref{EQ4DLS} by $\mathcal{G}$. It is direct to check that $\mathcal{G}\subset \D(\L^\s)$ and $\L^\s u=\frac{1}{2}\frac{d}{dm}\frac{du}{d\lambda}$ for $u\in \mathcal{G}$. To the contrary, take $u\in \D(\L^\s)$ with $\L^\s u=f\in H$. Then for any fixed $M>0$ and any $v\in \FF^\s\cap C_c(\mathbb{G})$ with $\text{supp}[v_\pm]\subset \mathbb{G}_\pm\cap [-M,M]$, 
\[
	\EE^\s(u,v)=(-f,v)_H. 
\]
On one hand, $v$ is of bounded variation  and we have
\[
	\EE^\s(u,v)=\frac{1}{2}\int_{\mathbb{G}_+}\frac{du}{d\lambda} dV+\frac{1}{2}\int_{\mathbb{G}_-}\frac{du}{d\lambda} dV +C(v(0+)-v(0-)),
\]
where $V$ is the signed measure induced by $v$ and $C:=\frac{\kappa}{4}(u(0+)-u(0-))$. 
On the other hand, write 
\[	F(x):=F(0\pm)+\int_{0\pm}^{(-M)\vee x \wedge M} f(x)m(dx),\quad x\in \mathbb{G}_\pm, 
\]
where $F(0+)$ and $F(0-)$ are two constants. Since $f\in L^2(\mathbb{G},m)$, it follows that $F$ is of bounded variation and $dF=fdm$ on $\mathbb{G}_\pm$ respectively. This implies 
\[
\begin{aligned}
	(-f,v)_H&=-\int_\mathbb{G_+}v(x)dF(x)-\int_{\mathbb{G}_-}v(x)dF(x)  \\
		&=\int_{\mathbb{G}_+}FdV+\int_{\mathbb{G}_-}FdV + F(0+)v(0_+)-F(0-)v(0-).
\end{aligned}
\]
By letting $v|_{\mathbb{G}_-}\equiv 0$ or $v|_{\mathbb{G}_+}\equiv 0$, we have
\[
	\frac{1}{2}\int_{\mathbb{G}_\pm}\frac{du}{d\lambda} dV\pm Cv(0\pm)=\int_{\mathbb{G}_\pm}FdV\pm F(0\pm)v(0\pm).
\]
Then we can easily conclude that $C=F(0\pm)$ and $\frac{1}{2}\frac{du}{d\lambda}=F$ on $\mathbb{G}$. This indicates 
\[
	\frac{du_\pm}{d\lambda_\pm}\ll m_\pm, \quad \frac{1}{2}\frac{d}{dm_\pm}\left(\frac{du_\pm}{d\lambda_\pm}\right)=f_\pm\in L^2(\mathbb{G}_\pm,m_\pm)
	\]
and 
\[
	\frac{du_\pm}{d\lambda_\pm}(0\pm)=2†C=\frac{\kappa}{2}(u(0+)-u(0-)).
\]
That completes the proof.
\end{proof}

The semigroup $P_t^\dagger$ of $X^\dagger$ satisfies the strong Feller property in the sense that 
\[
	P^\dagger_tf(\cdot)=\mathbf{E}_\cdot[f(X^\dagger_t)]\in C_b(E),\quad \forall f\in \mathcal{B}_b(E),
\]
where $E=\mathbb{R}$ or $\mathbb{G}$ is the state space of $X^\dagger$. Indeed, take $g_n\in L^2(E,m)$ with $g_n\uparrow 1$, and set $f_n:=f\cdot g_n$. Then $f_n\in L^2(E,m)$ and thus $P_t^\dagger f_n$ is a quasi-continuous function in $\FF^\dagger$. This indicates $P_t^\dagger f_n\in C_b(E)$, since every singleton of $E$ is of positive capacity with respect to $(\EE^\dagger, \FF^\dagger)$. Therefore we can conclude $P^\dagger_tf\in C_b(E)$ from $\|P^\dagger_tf_n-P^\dagger_tf\|_{C_b}\leq \|f_n-f\|_\infty\rightarrow 0$. The strong Feller property of $P^\dagger_t$ tells us it is also feasible to explore the generator of $X^\dagger$ on $C_b(E)$. We refer further considerations to \cite{F14}. 

\subsection{Phase transition of stiff problem}\label{SEC43}
 
As mentioned before, the stiff problem is concerned with a thermal conduction model with a singular barrier. In this subsection, we shall focus on the probabilistic description of this problem, and the main tool is the Mosco convergence of Dirichlet forms introduced in \S\ref{SEC41}. 

For $\varepsilon>0$, assume that a normal barrier is located at $I_\varepsilon=(-\varepsilon,\varepsilon)$. It is identified with a thermal resistance $\gamma_\varepsilon$ on $I_\varepsilon$. In other words, $\gamma_\varepsilon$ is a positive, finite and fully supported measure on $I_\varepsilon$ charging no set of singleton. Let $\mathbb{R}\setminus I_\varepsilon$ be of normal material with $T_{\varepsilon}^{\#}\lambda$ being its thermal resistance. Recall that $T_\varepsilon: \mathbb{G}\rightarrow \mathbb{R}\setminus I_\varepsilon$ is a homeomorphism, and $T_{\varepsilon}^{\#}\lambda$ is the image measure of $\lambda$ under $T_\varepsilon$. Set a measure on $\mathbb{R}$
\begin{equation}\label{EQ4LVT}
	\lambda_\varepsilon:= T_{\varepsilon}^{\#}\lambda+\gamma_\varepsilon. 
\end{equation}
Clearly, $\lambda_\varepsilon\in \mathscr{M}$ and denote its induced scale function by $\ss_\varepsilon$. By means of $m$ and $\lambda_\varepsilon$, we could write the Dirichlet form related to the thermal conduction model with the normal barrier $(I_\varepsilon, \gamma_\varepsilon)$ as follows
\begin{equation}\label{EQ4FFLR}
\begin{aligned}
&\FF^\varepsilon=\bigg\{f\in L^2(\mathbb{R}, m): f\ll \lambda_\varepsilon, \int_{\mathbb{R}} \left(\frac{df}{d\lambda_\varepsilon}\right)^2d\lambda_\varepsilon<\infty, \\
&\qquad\qquad\qquad\qquad \qquad f(\pm\infty):=\lim_{x\rightarrow \pm\infty}f(x)=0 \text{ if }\lambda_\varepsilon(\mathbb{G}_\pm)<\infty\bigg\}, \\
&\EE^\varepsilon(f,g)=\frac{1}{2}\int_{\mathbb{R}}\frac{df}{d\lambda_\varepsilon}\frac{dg}{d\lambda_\varepsilon}d\lambda_\varepsilon,\quad f,g\in \FF^\varepsilon. 
\end{aligned}
\end{equation} 
The associated diffusion $X^\varepsilon$ of $(\EE^\varepsilon, \FF^\varepsilon)$ is irreducible and $m$-symmetric on $\mathbb{R}$. 

The main purpose of this section is to study the convergence of $(\EE^\varepsilon,\FF^\varepsilon)$ as $\varepsilon\downarrow 0$. 
Before moving on, we need to prepare some notations. Take a decreasing sequence $\varepsilon_n\downarrow 0$ and write $I_n, \gamma_n, \lambda_n, (\EE^n, \FF^n)$ for $I_{\varepsilon_n}, \gamma_{\varepsilon_n}, \lambda_{\varepsilon_n}, (\EE^{\varepsilon_n}, \FF^{\varepsilon_n})$ respectively. 
Set
\[
	m^*(n):=\sup_{x\in \mathbb{R}}m\left([x,x+\varepsilon_n ]\right), \quad \lambda^*(n):=\sup_{x\in \mathbb{R}}\lambda\left([x,x+\varepsilon_n ]\right).
\]
Moreover, $\bar{\gamma}(n):=\gamma_n(I_n)$ is called the \emph{total thermal resistance} of $I_n$. In the following theorem, we build a phase transition in the context of the convergence of $(\EE^n,\FF^n)$ as $n\rightarrow \infty$. This phase transition sheds light on the patterns of thermal conduction model with a singular barrier at $0$, which definitely depend on its total thermal resistance. Notice that although the associated Markov processes live in $\mathbb{G}$ or $\mathbb{R}$, the Dirichlet forms \eqref{EQ4FFLG}, \eqref{EQ4FSF}, \eqref{EQ4FIF} and $(\EE^n,\FF^n)$ are on the same Hilbert space $H=L^2(\mathbb{G},m)=L^2(\mathbb{R},m)$. Thus $H$ is also the underlying space of Mosco convergences below. 

\begin{theorem}\label{THM45}
Let $\varepsilon_n, I_n, \gamma_n, \lambda_n, (\EE^n, \FF^n)$ be given above. Assume 
\begin{equation}\label{EQ4LNM}
	\bar{\gamma}(n)m^*(n)+\lambda^*(n)m^*(n)\rightarrow 0\quad \text{as }n\rightarrow \infty,
\end{equation}
and 
\[
	\bar{\gamma}:=\lim_{n\rightarrow \infty} \bar{\gamma}(n) \quad (\leq \infty)
\]
exists. Then the following assertions hold:
\begin{itemize}
\item[(1)] $\bar{\gamma}=\infty$: $(\EE^n, \FF^n)$ converges to the Dirichlet form $(\EE, \FF)$ given by \eqref{EQ4FFLG} in the sense of Mosco.
\item[(2)] $0<\bar{\gamma}<\infty$: $(\EE^n, \FF^n)$ converges to the Dirichlet form $(\EE^\s, \FF^\s)$ given by \eqref{EQ4FSF} with the parameter $\kappa=2/\bar{\gamma}$ in the sense of Mosco.
\item[(3)] $\bar{\gamma}=0$: $(\EE^n, \FF^n)$ converges to the Dirichlet form $(\EE^\i, \FF^\i)$ given by \eqref{EQ4FIF} in the sense of Mosco. 
\end{itemize}
\end{theorem}
\begin{proof}
\begin{itemize}
\item[(1)] Suppose $\{f_n\}$ converges to $f$ weakly in $H$ and $\liminf_{n\rightarrow\infty}\EE^n(f_n,f_n)<\infty$. For showing $\EE(f,f)\leq \liminf_{n\rightarrow\infty}\EE^n(f_n,f_n)$, there is no loss of generality in assuming
\[
	M:=\sup_{n\geq 1}\EE^n(f_n,f_n)<\infty. 
\]
Define a function $\breve{f}_n:=f_n\circ T_{\varepsilon_n}$, i.e. $\breve{f}_n(x):=f_n(x+\varepsilon_n)$ for $x\geq 0$ and $\breve{f}_n(x):=f_n(x-\varepsilon_n)$ for $x<0$. We assert 
\begin{equation}\label{EQ4FNFN}
	\|f_n-\breve{f}_n\|_H\rightarrow 0 \quad \text{as } n\rightarrow \infty, 
\end{equation}
and particularly, $\breve{f}_n$ converges to $f$ weakly in $H$. Indeed, 
\[
	\|f_n-\breve{f}_n\|_H^2=\int_0^\infty \left(f_n(x)-\breve{f}_n(x)\right)^2 m(dx) +\int_{-\infty}^0 \left(f_n(x)-\breve{f}_n(x)\right)^2 m(dx).
\]
We can deduce that
\[
\begin{aligned}
\int_0^\infty \left(f_n(x)-\breve{f}_n(x)\right)^2 m(dx)&=\int_0^\infty \left(\int_x^{x+\varepsilon_n}\frac{df_n}{d\lambda_n}d\lambda_n\right)^2m(dx) \\
&\leq (\lambda^*(n)+\bar{\gamma}(n))m^*(n)\cdot \int_0^\infty \left(\frac{df_n}{d\lambda_n}\right)^2d\lambda_n \\
&\leq 2M\cdot (\lambda^*(n)+\bar{\gamma}(n))m^*(n). 
\end{aligned}
\]
Similarly, $\int_{-\infty}^0 \left(f_n(x)-\breve{f}_n(x)\right)^2 m(dx)\leq 2M\cdot (\lambda^*(n)+\bar{\gamma}(n))m^*(n)$ and thus $\|f_n-\breve{f}_n\|_H^2\rightarrow 0$ by \eqref{EQ4LNM}. Clearly, $\breve{f}_n\in \FF$. Then it follows from $\breve{f}_n=f_n\circ T_{\varepsilon_n}$ and $\lambda_n|_{I^c_n}=\lambda\circ T^{-1}_{\varepsilon_n}$ that
\[
\begin{aligned}
	\EE(f,f)&\leq \liminf_{n\rightarrow \infty} \EE(\breve{f}_n,\breve{f}_n) \\
	&=\liminf_{n\rightarrow\infty} \frac{1}{2}\int_\mathbb{G} \left(\frac{d\breve{f}_n}{d\lambda}\right)^2d\lambda  \\
	&=\liminf_{n\rightarrow \infty} \frac{1}{2}\int_{I^c_n}\left(\frac{df_n}{d\lambda_n}\right)^2d\lambda_n \\
	&\leq \liminf_{n\rightarrow \infty} \EE^n(f_n,f_n). 
\end{aligned}\]

On the other hand, let $g\in H$ with $\EE(g,g)<\infty$. Particularly, $g$ is continuous on $\mathbb{G}_+$ and $\mathbb{G}_-$ respectively, and $g(0+), g(0-)$ are well defined. For each $n$, define a function $g_n$ as follows: 
\begin{equation}\label{EQ4GNI}
g_n|_{I^c_n}:=g\circ T^{-1}_{\varepsilon_n},\quad 	g_n(x):=g(0-)+c_n\cdot\int_{-\varepsilon_n}^x d\gamma_n, \quad x\in I_n,
\end{equation}
with $c_n:=(g(0+)-g(0-))/\bar{\gamma}(n)$. Clearly, $g_n\in \FF^n$. Since $\bar{\gamma}(n)\rightarrow \infty$,  we have
\[
	\EE^n(g_n,g_n)=\EE(g.g)+\frac{1}{2}c_n^2\cdot \bar{\gamma}(n)=\EE(g,g)+\frac{(g(0+)-g(0-))^2}{2\bar{\gamma}(n)}\rightarrow \EE(g,g).
\]
Mimicking \eqref{EQ4FNFN}, we can also obtain $\|g_n-g\|_H\rightarrow 0$. This implies $\{g_n\}$ is a sequence that converges to $g$ strongly in $H$ and
\[
	\limsup_{n\rightarrow \infty}\EE^n(g_n,g_n)\leq \EE(g,g). 
\]
\item[(2)] Suppose $\{f_n\}$ converges to $f$ weakly in $H$, $\liminf_{n\rightarrow\infty}\EE^n(f_n,f_n)<\infty$ and $M:=\sup_{n\geq 1}\EE^n(f_n,f_n)<\infty$. Let $\breve{f}_n=f_n\circ T_{\varepsilon_n}\in \FF=\FF^\s$. We know that $\breve{f}_n\rightarrow f$ weakly in $H$. Since 
\begin{equation}\label{EQ4FNF}
\begin{aligned}
	\left(\breve{f}_n(0+)-\breve{f}_n(0-)\right)^2 &=\left(f_n(\varepsilon_n)-f_n(-\varepsilon_n)\right)^2 \\
	&=\left(\int_{-\varepsilon_n}^{\varepsilon_n}\frac{df_n}{d\gamma_n}d\gamma_n  \right)^2 \\
	&\leq \bar{\gamma}(n)\int_{-\varepsilon_n}^{\varepsilon_n}\left(\frac{df_n}{d\gamma_n}\right)^2d\gamma_n,
\end{aligned}\end{equation}
it follows that
\[
\begin{aligned}
	\EE^\s(f,f)&\leq \liminf_{n\rightarrow \infty} \EE^\s(\breve{f}_n,\breve{f}_n) \\
	&=\liminf_{n\rightarrow \infty}\left(\frac{1}{2}\int_\mathbb{G}\left(\frac{d\breve{f}_n}{d\lambda}\right)^2d\lambda+\frac{\kappa}{4}\left(\breve{f}_n(0+)-\breve{f}_n(0-)\right)^2\right) \\
	&\leq \liminf_{n\rightarrow \infty}\left(\frac{1}{2}\int_{I^c_n}\left(\frac{df_n}{d\lambda_n}\right)^2d\lambda_n+ \frac{\bar{\gamma}(n)\kappa}{4}\int_{-\varepsilon_n}^{\varepsilon_n}\left(\frac{df_n}{d\gamma_n}\right)^2d\gamma_n \right) \\
	&= \liminf_{n\rightarrow \infty}\left(\EE^n(f_n,f_n)+\frac{\bar{\gamma}(n)\kappa-2}{4}\int_{-\varepsilon_n}^{\varepsilon_n}\left(\frac{df_n}{d\gamma_n}\right)^2d\gamma_n \right).
\end{aligned}\]
Note that $\int_{-\varepsilon_n}^{\varepsilon_n}\left(\frac{df_n}{d\gamma_n}\right)^2d\gamma_n\leq \EE^n(f_n,f_n)\leq M$. As a consequence,
\[
	\lim_{n\rightarrow\infty}\left|\frac{\bar{\gamma}(n)\kappa-2}{4}\int_{-\varepsilon_n}^{\varepsilon_n}\left(\frac{df_n}{d\gamma_n}\right)^2d\gamma_n\right| \leq M\lim_{n\rightarrow\infty}\left|\frac{\bar{\gamma}(n)\kappa-2}{4}\right|=0.
\]
This implies $\EE^\s(f,f)\leq \liminf_{n\rightarrow \infty}\EE^n(f_n,f_n)$. 

On the other hand, let $g\in H$ with $\EE^\s(g,g)<\infty$. Take $g_n$ as in \eqref{EQ4GNI}. Then $g_n\in \FF^n$ and 
\[
	\EE^n(g_n,g_n)=\EE(g,g)+\frac{(g(0+)-g(0-))^2}{2\bar{\gamma}(n)}\rightarrow \EE^\s(g,g). 
\]
Similar to \eqref{EQ4FNFN}, we can also conclude $\lim_{n\rightarrow \infty}\|g_n-g\|_H=0$.
\item[(3)] We still suppose $\{f_n\}$ converges to $f$ weakly in $H$, $\liminf_{n\rightarrow\infty}\EE^n(f_n,f_n)<\infty$ and $M:=\sup_{n\geq 1}\EE^n(f_n,f_n)<\infty$. It has been proved in the case $\bar{\gamma}=\infty$ that $f\in \FF$ and
\[
	\EE(f,f)\leq \liminf_{n\rightarrow \infty}\EE^n(f_n,f_n). 
\]
We need only show $f\in \FF^\i$, which implies $\EE^\i(f,f)=\EE(f,f)\leq \liminf_{n\rightarrow \infty}\EE^n(f_n,f_n)$. In fact, $f$ is continuous on $\mathbb{G}_+$ and $\mathbb{G}_-$ respectively. We still consider $\breve{f}_n=f_n\circ T_{\varepsilon_n}$. Clearly, $\breve{f}_n\rightarrow f$ weakly in $H$ and $\sup_{n}\EE(\breve{f}_n,\breve{f}_n)\leq \sup_n \EE^n(f_n,f_n)\leq M$. The weak convergence of $\breve{f}_n$ in $H$ implies $\sup_n\|\breve{f}_n\|_H<\infty$. Thus $\sup_{n}\EE_1(\breve{f}_n,\breve{f}_n)<\infty$. By Banach-Saks theorem, the Ces\`aro mean of a suitable subsequence of $\{\breve{f}_n\}$ converges to some $h\in \FF$ in $\|\cdot \|_{\EE_1}$-norm. Without loss of generality, we still denote this subsequence by $\{\breve{f}_n\}$. Then $h_k:=\frac{1}{k}\sum_{n=1}^k\breve{f}_n$ is $\EE_1$-convergent to $h$. This implies $h_k$ converges to $h$, $\EE$-q.e., and particularly, $h_k(0\pm)\rightarrow h(0\pm)$. It follows from \eqref{EQ4FNF} that $|\breve{f}_n(0+)-\breve{f}_n(0-)|\leq \sqrt{M\cdot \bar{\gamma}(n)}\rightarrow 0$ as $n\rightarrow \infty$. Hence
\[
	|h(0+)-h(0-)|=\lim_{k\rightarrow \infty}\left|\frac{1}{k}\sum_{n=1}^k\left(\breve{f}_n(0+)-\breve{f}_n(0-)\right) \right|=0. 
\]
This indicates $h$ is continuous on $\mathbb{R}$, and so that $h\in \FF^\i$. Take any $u\in H$, we have $(\breve{f}_n, u)_H\rightarrow (f, u)_H$ and 
\[
(h,u)_H=\lim_{k\rightarrow \infty} (h_k, u)_H=\lim_{k\rightarrow \infty}\frac{1}{k}\sum_{n=1}^k (\breve{f}_n, u)_H=(f,u)_H.
\]
Therefore, $f=h\in \FF^\i$. 

On the other hand, let $g\in H$ with $\EE^\i(g,g)<\infty$. This means $g\in \FF^\i$ and $g$ is continuous on $\mathbb{R}$. Consider $g_n$ in \eqref{EQ4GNI}. Note that $g_n(x)=g(0)$ for any $x\in [-\varepsilon_n,\varepsilon_n]$ since $c_n=0$. Clearly, $g_n\rightarrow g$ strongly in $H$ and 
\[
	\EE^n(g_n,g_n)=\EE(g,g). 
\]
\end{itemize}
That completes the proof. 
\end{proof}
\begin{remark}
In \cite{L16}, $m$ and $\lambda$ are both the Lebesgue measure, and $\gamma_\varepsilon$ is taken to be $\frac{dx}{\kappa\varepsilon}$ on $I_\varepsilon$. Clearly, \eqref{EQ4LNM} holds and $\bar{\gamma}=\bar{\gamma}_\varepsilon(I_\varepsilon)=2/\kappa$. The snapping out Markov process associated with the limit of $(\EE^\varepsilon, \FF^\varepsilon)$ as $\varepsilon \downarrow 0$ is actually the SNOB with the parameter $\kappa$. 
\end{remark}

We call the three patterns of thermal conduction in Theorem~\ref{THM45}
\begin{itemize}
\item[(1)] the \emph{impermeable pattern} for the phase $\bar{\gamma}=\infty$,
\item[(2)] the \emph{semi-permeable pattern} for the phase $0<\bar{\gamma}<\infty$, and
\item[(3)] the \emph{permeable pattern} for the phase $\bar{\gamma}=0$.
\end{itemize}    
The most interesting case is the semi-permeable pattern (it is very similar to the `barrier penetration' in quantum mechanics). As we have shown in \S\ref{SEC3}, the penetrations in this case are realized by additional jumps between $0+$ and $0-$ in the probabilistic counterpart.
The parameter $\kappa$, i.e. the reciprocal of total thermal resistance, reflects the ability of the flow to penetrate the singular barrier.

Though the convergences in Theorem~\ref{THM45} are in the manner of Dirichlet forms, we can also obtain the convergences of corresponding Markov processes in the sense of finite dimensional distributions. Let $(\EE^n,\FF^n)$ be in Theorem~\ref{THM45} (or Corollary~\ref{COR48}) and $X^n$ be its associated diffusion on $\mathbb{R}$. Further let $(\EE^\dagger, \FF^\dagger)$ be one of $(\EE,\FF)$, $(\EE^\s,\FF^\s)$ and $(\EE^\i,\FF^\i)$ and denote its associated Markov process by $X^\dagger=(X^\dagger_t)_{t\geq 0}$. Write $(\mathbf{P}_x^n)_{x\in \mathbb{R}}$, $(\mathbf{P}_x)_{x\in E}$ ($E=\mathbb{R}$ or $\mathbb{G}$) for the probability measures of $X^n$ and $X^\dagger$ respectively. Take a function $h\in L^2(\mathbb{R},m)=L^2(\mathbb{G},m)$ and set
\[
	\mathbf{P}^n_{h\cdot m}[\cdot]:= \int_\mathbb{R}h(x)m(dx)\mathbf{P}^n_x[\cdot], \quad \mathbf{P}_{h\cdot m}[\cdot]:= \int_\mathbb{R}h(x)m(dx)\mathbf{P}_x[\cdot]. 
\]
The expectation with respect to $\mathbf{P}^n_{h\cdot m}$ (resp. $\mathbf{P}_{h\cdot m}$) is denoted by $\mathbf{E}^n_{h\cdot m}$ (resp. $\mathbf{E}_{h\cdot m}$). 
Then the following result holds. The proof is direct by using Proposition~\ref{PRO42}, see  \cite[Proposition~4.3]{LUY17}. 

\begin{corollary}\label{COR49}
Assume $(\EE^n,\FF^n)$ converges to $(\EE^\dagger, \FF^\dagger)$ in the sense of Mosco and fix $h\in L^2(\mathbb{R},m)$. Then for any $k\geq 1$, $0\leq t_1<\cdots<t_k<\infty$ and $f_i\in \mathcal{B}_b(\mathbb{R})\cap L^2(\mathbb{R},m)$ with $1\leq i\leq k$, it holds that
\begin{equation}\label{EQ4NEN}
	\lim_{n\rightarrow \infty} \mathbf{E}^n_{h\cdot m}\left[f_1(X^n_{t_1})\cdots f_k(X^n_{t_k})\right]=\mathbf{E}_{h\cdot m}\left[f_1(X^\dagger_{t_1})\cdots f_k(X^\dagger_{t_k})\right]. 
\end{equation}
\end{corollary}
\begin{remark}
In the case of impermeable pattern or semi-permeable pattern, $E=\mathbb{G}$. Thus $f_i$ should be replaced by a suitable measurable function $\breve{f}_i$ on $\mathbb{G}$ in the right side of \eqref{EQ4NEN}. Clearly, $f_i=\breve{f}_i$ apart from $0$ (or $0\pm$). Thanks to \cite[Theorem~4.2.3]{FOT11}, we know that
\[
	\mathbf{E}_{h\cdot m}\left[\breve{f}_1(X^\dagger_{t_1})\cdots \breve{f}_k(X^\dagger_{t_k})\right]=\mathbf{E}_{h\cdot m}\left[\hat{f}_1(X^\dagger_{t_1})\cdots \hat{f}_k(X^\dagger_{t_k})\right],
\]
if $\hat{f}_i$ is another appropriate version of $f_i$ on $\mathbb{G}$, i.e. $\hat{f}_i(x)=f_i(x)$ for $x\neq 0$. So in abuse of symbols, we still use $f_i$ in the right side of \eqref{EQ4NEN}.  

On the other hand, the convergence in \eqref{EQ4NEN}  is weaker than the weak convergence of $\{\mathbf{P}^n_{h\cdot m}: n\geq 1\}$, by realizing which as  a family of probability measures (suppose $\int hdm=1$) on the space $C([0,\infty), \mathbb{R})$ of continuous paths or Skorokhod space $D([0, \infty),\mathbb{R})$ of c\`adl\`ag paths. However for the weak convergence, we are stuck in the trouble that $X^\dagger$ might live in $\mathbb{G}$ and $C([0,\infty), \mathbb{G})$ (resp. $D([0, \infty),\mathbb{G})$) differs from $C([0,\infty), \mathbb{R})$ (resp. $D([0, \infty),\mathbb{R})$) significantly.
\end{remark}

Let us briefly explain the technical condition \eqref{EQ4LNM} in Theorem~\ref{THM45}. As mentioned in Remark~\ref{RM43}, $m$ is usually taken to be the Lebesgue measure in the thermal conduction. Without loss of generality, we take $\varepsilon_n=1/n$ further. Then the first part of \eqref{EQ4LNM} becomes
\begin{equation}\label{EQ4NIG}
	\lim_{n\rightarrow \infty} \frac{\bar{\gamma}(n)}{n}=0. 
\end{equation}
It has no effects on the semi-permeable and permeable patterns. However, in the impermeable pattern, \eqref{EQ4NIG} causes that the divergence of $\bar{\gamma}(n)$ must be slower than $n$. We believe this restriction is not essential. Indeed, the convergence of the phase $\bar{\gamma}=\infty$ is proved for the Brownian case  without \eqref{EQ4NIG} in Corollary~\ref{COR48}. On the other hand, the second part of \eqref{EQ4LNM} is
\begin{equation}\label{EQ4NLN}
	\lim_{n\rightarrow \infty}\frac{\lambda^*(n)}{n}=0. 
\end{equation}
This is a very mild assumption. It admits $\lambda$ to be not absolutely continuous. For example, let 
\begin{equation}\label{EQ4LDX}
	d\lambda=dx+d\mathfrak{c},
\end{equation}
where $\mathfrak{c}$ is the Cantor function with $\mathfrak{c}(x)=0$ for any $x\leq 0$ and $\mathfrak{c}(x)=1$ for any $x\geq 1$. Then $\lambda^*(n)\leq 1$ and \eqref{EQ4NLN} holds, but $\lambda$ is not absolutely continuous. When $\lambda$ is absolutely continuous, write $a(x):=1/\ss'(x)$ for the thermal conductivity. Since 
\[
	\lambda\left(\left[y,y+\frac{1}{n}\right]\right)=\int_y^{y+\frac{1}{n}}\frac{1}{a(x)}dx,
\]
we find that the condition $a(x)\geq \delta$ a.e. with some constant $\delta>0$ implies \eqref{EQ4NLN}. But \eqref{EQ4NLN} also admits $a$ to be very close to $0$. For example, take $0<\beta<1$ and 
\begin{equation}\label{EQ4AXX}
	a(x)=|x|^\beta\wedge 1,\quad x\in \mathbb{R}.
\end{equation}
Then $\lambda(dx):=\frac{1}{a(x)}dx$ satisfies \eqref{EQ4NLN}.
 
\subsection{Brownian case of phase transition}
The short subsection is to present the Brownian case of Theorem~\ref{THM45}, in which the phase transition becomes more complete.

\begin{corollary}\label{COR48}
Let $m$ and $\lambda$ be the Lebesgue measure. Then the assertions in Theorem~\ref{THM45} hold without the condition \eqref{EQ4LNM}. Particularly, take $\alpha\in \mathbb{R}, \kappa>0$ and set 
\[
	\gamma_n(dx)=(\kappa\varepsilon_n)^\alpha dx. 
\]
Then we have:
\begin{itemize}
\item[(1)] $\alpha<-1$: $(\EE^n,\FF^n)$ converges to the Dirichlet form \eqref{EQ3FUL} of two-sided reflecting Brownian motion on $\mathbb{G}$ in the sense of Mosco.
\item[(2)] $\alpha=-1$: $(\EE^n,\FF^n)$ converges to the Dirichlet form \eqref{EQ3FSUL} of snapping out Brownian motion on $\mathbb{G}$ with the parameter $\kappa$ in the sense of Mosco.
\item[(3)] $\alpha>-1$: $(\EE^n,\FF^n)$ converges to the Dirichlet form $(\frac{1}{2}\mathbf{D}, H^1(\mathbb{R}))$ of one-dimensional  Brownian motion on $\mathbb{R}$ in the sense of Mosco.
\end{itemize}
\end{corollary}
\begin{proof}
Note that $\gamma_n(I_n)=2\kappa^\alpha\cdot \varepsilon_n^{\alpha+1}$. Thus it suffices to prove the case $\bar{\gamma}=\infty$ without \eqref{EQ4LNM}. We still denote the Dirichlet form \eqref{EQ3FUL} by $(\EE,\FF)$. Suppose $\{f_n\}$ converges to $f$ weakly in $H=L^2(\mathbb{R})=L^2(\mathbb{G})$ and $\liminf_{n\rightarrow \infty}\EE^n(f_n,f_n)\leq \sup_{n}\EE^n(f_n,f_n)=M<\infty$. Set
\[
	\breve{f}_n(x):=\left\lbrace \begin{aligned}
		&f_n(-\varepsilon_n),\quad x\in (-\varepsilon_n, 0-],\\
		&f_n(\varepsilon_n),\quad x\in [0+,\varepsilon_n),\\
		 &f_n(x),\quad x\in I_n^c. 
	\end{aligned}\right.
\]
Clearly, $\breve{f}_n\in \FF$. For any $g\in H$, we have
\[
	\left(f_n-\breve{f}_n, g\right)_{H}=\int_{I_n} f_n(x)g(x)dx-f_n(-\varepsilon_n)\int_{-\varepsilon_n}^0g(x)dx-f_n(\varepsilon_n)\int_0^{\varepsilon_n}g(x)dx. 
\]
The weak convergence of $\{f_n\}$ implies $K:=\sup_{n}\|f_n\|^2_{H}<\infty$. Thus as $n\rightarrow \infty$, 
\[
	\left|\int_{I_n} f_n(x)g(x)dx\right|^2\leq K\cdot \int_{I_n}g(x)^2dx\rightarrow 0. 
\]
Since $|f_n(\varepsilon_n)|\leq |f_n(x)|+|\int_{\varepsilon_n}^xf'_n(y)dy|$ for any $x>\varepsilon_n$, it follows that
\begin{equation}\label{EQ4FNN}
	|f_n(\varepsilon_n)|^2\leq 2\int_{\varepsilon_n}^{1+\varepsilon_n} f_n(x)^2dx+2\int_{\varepsilon_n}^\infty f'_n(y)^2dy\leq 4\EE^n_1(f_n,f_n). 
\end{equation}
This implies $\left|f_n(\varepsilon_n)\int^{\varepsilon_n}_0g(x)dx\right|\leq 4(M+K)\left|\int_{-\varepsilon_n}^0g(x)dx\right|\rightarrow 0$ as $n\rightarrow \infty$. Similarly, we can deduce that $f_n(-\varepsilon_n)\int_{-\varepsilon_n}^0g(x)dx\rightarrow 0$. As a consequence, we can conclude $\breve{f}_n$ converges to $f$ weakly in $H$, and
\[
\EE(f,f)\leq \liminf_{n\rightarrow \infty} \EE(\breve{f}_n,\breve{f}_n)\leq\liminf_{n\rightarrow \infty}\EE^n(f_n,f_n). 
\]
On the other hand, let $g\in H$ with $\EE(g,g)<\infty$. For each $n$, define a function $g_n$ as follows:
\[
g_n|_{I^c_n}:=g|_{I^c_n},\quad 	g_n(x):=g(-\varepsilon_n)+c_n\cdot\int_{-\varepsilon_n}^x d\gamma_n, \quad x\in I_n,
\]
where $c_n:=(g(\varepsilon_n)-g(-\varepsilon_n))/\bar{\gamma}(n)$. We assert $\|g_n-g\|_H\rightarrow 0$ as $n\rightarrow \infty$. In fact, mimicking \eqref{EQ4FNN}, we can obtain for any $x\in I_n$, 
\[
|g_n(x)|^2\leq |g( \varepsilon_n)|^2+ |g(-\varepsilon_n)|^2\leq 8\EE_1(g,g). 
\]
It follows that
\[
\|g_n-g\|_H^2\leq 2\int_{I_n}g(x)^2dx+2\int_{I_n}g_n(x)^2dx\leq 2\int_{I_n}g(x)^2dx+ 32\EE_1(g,g)\cdot \varepsilon_n\rightarrow 0. 
\]
Clearly, $g_n\in \FF^n$ and we can deduce that
\[
\begin{aligned}
	\EE^n(g_n,g_n)&=\frac{1}{2}\int_{I^c_n}g'(x)^2dx+\frac{1}{2}c_n^2\cdot \bar{\gamma}(n) \\
		&\leq \frac{1}{2}\int_{I^c_n}g'(x)^2dx+\frac{32\EE_1(g,g)^2}{\bar{\gamma}(n)}. 
\end{aligned}
\]
Since $\bar{\gamma}(n)\rightarrow \infty$, we can conclude 
\[
	\limsup_{n\rightarrow \infty}\EE^n(g_n,g_n)\leq \limsup_{n\rightarrow \infty}\frac{1}{2}\int_{\mathbb{G}}g'(x)^2dx=\EE(g,g). 
\]
That completes the proof. 
\end{proof}

\subsection{Continuity of the phase transition}

This subsection is to derive the continuity of the phase transition in Theorem~\ref{THM45} in the sense that the Dirichlet forms, which describe the phases, are continuous in the parameter $\bar{\gamma}$.  

To show this continuity,  let us make some notations for convenience. For any $\bar{\gamma}\in [0,\infty]$, write $(\mathcal{E}^{\bar{\gamma}},\mathcal{F}^{\bar{\gamma}})$ for the limit in Theorem~\ref{THM45} that corresponds to the total resistance $\bar{\gamma}$. In other words, 
\begin{itemize}
\item[(1)] $(\mathcal{E}^\infty,\mathcal{F}^\infty):=(\EE,\FF)$ given by \eqref{EQ4FFLG};
\item[(2)] $(\mathcal{E}^{\bar{\gamma}},\mathcal{F}^{\bar{\gamma}}):=(\EE^\s,\FF^\s)$ given by \eqref{EQ4FSF} with $\kappa=2/\bar{\gamma}$ for any $0<\bar{\gamma}<\infty$;
\item[(3)] $(\mathcal{E}^0,\mathcal{F}^0):=(\EE^\i,\FF^\i)$ given by \eqref{EQ4FIF}.
\end{itemize}
The following result states that $[0,\infty]\ni\bar{\gamma} \mapsto (\mathcal{E}^{\bar{\gamma}}, \mathcal{F}^{\bar{\gamma}})$ is continuous in the sense of Mosco.

\begin{theorem}
Let $\{\bar{\gamma}^l: l\geq 1\}$ be a sequence in $[0,\infty]$ such that $\lim_{l\rightarrow\infty} \bar{\gamma}^l=\bar{\gamma}\in [0,\infty]$. Then $(\mathcal{E}^{\bar{\gamma}^l},\mathcal{F}^{\bar{\gamma}^l})$ is convergent to $(\mathcal{E}^{\bar{\gamma}},\mathcal{F}^{\bar{\gamma}})$ in the sense of Mosco as $l\rightarrow \infty$. 
\end{theorem}
\begin{proof}
We still denote $H:=L^2(\mathbb{R},m)=L^2(\mathbb{G},m)$ and for the sake of brevity, write $(\mathcal{E}^l, \mathcal{F}^l)$ for $(\mathcal{E}^{\bar{\gamma}^l},\mathcal{F}^{\bar{\gamma}^l})$. Without loss of generality, we could assume $0<\bar{\gamma}^l<\infty$ for any $l$. 

Firstly, we consider the case $\bar{\gamma}=\infty$. To prove the first item of Definition~\ref{DEF41}, suppose $\{u_l: l\geq 1\}$ converges weakly to $u\in H$ and 
\begin{equation}\label{EQ4LEL}
	\lim_{l\rightarrow \infty}\mathcal{E}^l(u_l,u_l)\leq \sup_{l}\cE^l(u_l,u_l)\leq M<\infty.
\end{equation}
Then it follows from $u_l\in \cF^l=\cF^\infty=\FF$ that 
\begin{equation}\label{EQ4EUU}
	\cE^\infty(u,u)\leq \liminf_{l\rightarrow \infty} \cE^\infty(u_l,u_l)\leq \liminf_{l\rightarrow\infty} \cE^l(u_l,u_l). 
\end{equation}
For the second item of Definition~\ref{DEF41}, let $u\in H$ with $\cE^\infty(u,u)<\infty$. Take $u_l:=u\in \cF^\infty=\FF=\cF^l$, we have
\[
\lim_{l\rightarrow \infty}\cE^l(u_l,u_l)=\lim_{l\rightarrow \infty}\cE^l(u,u)=\EE(u,u)=\cE^\infty(u,u). 
\] 

Secondly, we prove the case $0<\bar{\gamma}<\infty$. The second item of Definition~\ref{DEF41} could be checked by taking $u_l:=u$ as in the case $\bar{\gamma}=\infty$. It suffices to check the first item. Let $u_l$ and $u$ such that \eqref{EQ4LEL} holds. Take a constant $K>\bar{\gamma}$. Then for some integer $N$, $\bar{\gamma}^l<K$ for all $l>N$. It follows that
\[
\frac{1}{2K}\sup_{l>N}\left(u_l(0+)-u_l(0-)\right)^2\leq \sup_{l>N} \frac{1}{2\bar{\gamma}^l}\left(u_l(0+)-u_l(0-)\right)^2\leq \sup_{l>N}\cE^l(u_l,u_l)\leq M. 
\]
This implies $\sup_{l>N}\left(u_l(0+)-u_l(0-)\right)^2\leq 2KM$. As a consequence,
\[
\begin{aligned}
\cE^{\bar{\gamma}}(u,u)&\leq \liminf_{l\rightarrow \infty}\cE^{\bar{\gamma}}(u_l,u_l) \\
&=\liminf_{l>N,l\rightarrow\infty} \left(\cE^l(u_l,u_l)+\left(\frac{1}{2\bar{\gamma}}-\frac{1}{2\bar{\gamma}^l}\right)\cdot \left(u_l(0+)-u_l(0-)\right)^2 \right) \\
&=\liminf_{l>N,l\rightarrow\infty} \cE^l(u_l,u_l) \\
&=\liminf_{l\rightarrow\infty} \cE^l(u_l,u_l). 
\end{aligned}\]

Finally, let us consider the case $\bar{\gamma}=0$. For the second item of Definition~\ref{DEF41}, it is also very clear by taking $u_l:=u$, since $\cF^0=\FF^i\subset \cF^l$ for every $l$. For the first item, we still assume $u_l$ and $u$ satisfy \eqref{EQ4LEL}. We need only prove $u\in \cF^0$, which leads to 
\[
	\cE^0(u,u)=\cE^\infty(u,u)\leq \liminf_{l\rightarrow\infty}\cE^\infty(u_l,u_l)\leq \liminf_{l\rightarrow\infty}\cE^l(u_l,u_l). 
\]
This can be attained by mimicking the proof of the third assertion in Theorem~\ref{THM45}. That completes the proof.
\end{proof}
\begin{remark}
Note that $(\cE^0,\cF^0)$ is the darning of $(\cE^{\bar{\gamma}},\cF^{\bar{\gamma}})$ obtained by shorting $\{0+,0-\}$ into $0$ for any $\bar{\gamma}\in (0,\infty]$. The same fasion of Mosco convergence as the case $\bar{\gamma}=0$ was also considered in \cite{CP17} for the study of general darning transform. 
\end{remark}

\section{Boundary conditions of the flux at the barrier}\label{SEC44}

In this section, we shall consider the stiff problems in the context of heat equations in $\mathbb{R}$. Especially, the boundary conditions of the flux at the barrier will be derived for the three phases by means of Dirichlet forms.

\subsection{Heat equation with a normal barrier}

Take a function $a$ on $\mathbb{R}$ such that for some constants $\delta, C>0$,
\begin{equation}\label{EQ4ADA}
	\delta\leq a(x)\leq C,\quad \; \text{ a.e. }x\in \mathbb{R}.
\end{equation}
For any $\varepsilon>0$, let $b_\varepsilon$ be a function on $I_\varepsilon=(-\varepsilon,\varepsilon)$ such that for some constants $\delta_\varepsilon, C_\varepsilon>0$, 
\begin{equation}\label{EQ4BVD}
\delta_\varepsilon\leq b_\varepsilon(x) \leq C_\varepsilon \quad \; \text{ a.e. }x\in I_\varepsilon. 
\end{equation}
Set
\[
	a_\varepsilon(x)	:=\left\lbrace
	\begin{aligned}
	&a(x-\varepsilon),\quad x\geq \varepsilon, \\
	&b_\varepsilon(x),\quad x\in (-\varepsilon, \varepsilon), \\
	&a(x+\varepsilon),\quad x\leq -\varepsilon,
	\end{aligned}\right.
\]
and the stiff problem is concerned with the convergence of $u^\varepsilon$ (as $\varepsilon\downarrow 0$) in the heat equation
\begin{equation}\label{EQ4UVT}
\begin{aligned}
&\frac{\partial u^\varepsilon}{\partial t}(t,x)=\frac{1}{2}\nabla\left(a_\varepsilon(x)\nabla u^\varepsilon(t,x) \right),\quad t\geq 0,x\in \mathbb{R}, \\
&u^\varepsilon(0,\cdot)=u_0.
\end{aligned}
\end{equation}
The solution to \eqref{EQ4UVT} is considered to be a weak form as follows. 

\begin{definition}\label{DEF411}
A function $u^\varepsilon\in C_b\left([0,\infty), L^2(\mathbb{R})\right)\cap L^\infty\left([0,\infty),H^1(\mathbb{R})\right)$ is called a weak solution to \eqref{EQ4UVT} if $u^\varepsilon(0,\cdot)=u_0$, and for any $t>0$, $g\in C_c^\infty(\mathbb{R})$, 
\begin{equation}\label{EQ4RFX}
\int_\mathbb{R}(u_0(x)-u^\varepsilon(t,x))g(x)dx=\frac{1}{2}\int_0^t \int_\mathbb{R} a_\varepsilon(x)\nabla u^\varepsilon(s,x)\nabla g(x)dxds. 
\end{equation}
\end{definition}  

Though the well posedness of \eqref{EQ4UVT} is well known, we shall derive it by means of Dirichlet forms. Write 
\begin{equation}\label{EQ4LVX}
	\lambda_\varepsilon(dx):=\frac{1}{a_\varepsilon(x)}dx.
\end{equation}
Then \eqref{EQ4ADA} and \eqref{EQ4BVD} imply $\lambda_\varepsilon\in \mathscr{M}$ and denote its induced scale function by $\ss_\varepsilon$. Let $(\EE^\varepsilon,\FF^\varepsilon)$ be the Dirichlet form of the diffusion $X^\varepsilon$ with scale function $\ss_\varepsilon$. In other words, $(\EE^\varepsilon,\FF^\varepsilon)$ is \eqref{EQ4FFLR} with $\lambda_\varepsilon$ in \eqref{EQ4LVX} and $m$ being the Lebesgue measure on $\mathbb{R}$. Thanks to \cite[Theorem~3.2]{LY172}, $C_c^\infty(\mathbb{R})$ is a core of $(\EE^\varepsilon,\FF^\varepsilon)$ and for any $u,v\in \FF^\varepsilon$,
\[
	\EE^\varepsilon(u,v)=\frac{1}{2}\int_\mathbb{R} a_\varepsilon(x) u'(x)v'(x)dx. 
\]
Note that $\FF^\varepsilon= H^1(\mathbb{R})$ on account of $\delta \wedge \delta_\varepsilon\leq a_\varepsilon\leq C\vee C_\varepsilon$. Denote the semigroup of $X^\varepsilon$ by $(P^\varepsilon_t)_{t\geq 0}$. The following result claims the well posedness of \eqref{EQ4UVT}. 

\begin{lemma}\label{LM412}
Assume $u_0\in H^1(\mathbb{R})$. Then $u^\varepsilon(t,x):=P_t^\varepsilon u_0(x)$ is the unique weak solution to \eqref{EQ4UVT}. 
\end{lemma}
\begin{proof}


Let $(\EE^\varepsilon,\FF^\varepsilon)$ be above. Note that
\[
	u^\varepsilon_t(x)=P_t^\varepsilon u_0(x)=\mathbf{E}_x u_0(X^\varepsilon_t). 
\]
We assert $u^\varepsilon$ is a weak solution to \eqref{EQ4UVT}. Indeed, $u^\varepsilon_t=P_t^\varepsilon u_0\in \FF^\varepsilon = H^1(\mathbb{R})$ and clearly, $\|P_t^\varepsilon u_0\|_{L^2(\mathbb{R})}\leq \|u_0\|_{L^2(\mathbb{R})}$ and $t\mapsto \|P_t^\varepsilon  u_0\|_{L^2(\mathbb{R})}$ is continuous. Moreover, since $u_0\in H^1(\mathbb{R})= \FF^\varepsilon$, it follows from \cite[Lemma~1.3.3]{FOT11} that
\[
	\|u_t^\varepsilon\|_{H^1(\mathbb{R})}^2\leq \|u_0\|_{L^2(\mathbb{R})}^2+\frac{1}{\delta\wedge \delta_\varepsilon}\EE^\varepsilon(P_t^\varepsilon u_0, P_t^\varepsilon u_0)\leq \|u_0\|_{L^2(\mathbb{R})}^2+\frac{1}{\delta\wedge \delta_\varepsilon}\EE^\varepsilon(u_0,u_0).
\]
For any $g\in C_c^\infty(\mathbb{R})\subset \FF^\varepsilon$, we have
\[
	\left|\int_\mathbb{R} a_\varepsilon(x)(u^\varepsilon_s)'(x)g'(x)dx\right|=2\left|\EE^\varepsilon(P_su_0,g)\right|\leq 2\EE^\varepsilon(u_0,u_0)^{1/2}\cdot \EE^\varepsilon(g,g)^{1/2}. 
\] 
This indicates 
\begin{equation*}
	s\mapsto \int_\mathbb{R} a_\varepsilon(x)\nabla u^\varepsilon(s,x)\nabla g(x)dx
\end{equation*}
is locally integrable in $[0,\infty)$ and particularly, both the left side $\mathscr{L}_t$ and right side $\mathscr{R}_t$ of \eqref{EQ4RFX} are continuous in $t$.  Denote the resolvent of $X^\varepsilon$ by $(R^\varepsilon_\alpha)_{\alpha>0}$. Clearly, for any $\alpha>0$,
\begin{equation}\label{EQ4UGL}
	(u_0,g)_{L^2(\mathbb{R})}-\alpha(R^\varepsilon_\alpha u_0, g)_{L^2(\mathbb{R})}=\EE^\varepsilon(R^\varepsilon_\alpha u_0,g)=\int_0^\infty \mathrm{e}^{-\alpha t}\EE^\varepsilon(P_t^\varepsilon  u_0, g)dt.
\end{equation}   
This implies 
\[
	\int_0^\infty \mathrm{e}^{-\alpha t}\mathscr{L}_tdt=\int_0^\infty \mathrm{e}^{-\alpha t}\mathscr{R}_tdt,
\]	
and so that $\mathscr{L}_t=\mathscr{R}_t$ for any $t>0$ in the light of their continuities. 

We turn to prove the uniqueness. Suppose $u_0=0$ and $u^\varepsilon$ is a weak solution to \eqref{EQ4UVT}. Then $u^\varepsilon_t\in H^1(\mathbb{R})=\FF^\varepsilon$ and for any $g\in C_c^\infty(\mathbb{R})$, 
\begin{equation}\label{EQ4RUT}
-\int_\mathbb{R}u^\varepsilon_t(x)g(x)dx=\int_0^t \EE^\varepsilon(u^\varepsilon_s, g) ds.
\end{equation}
Note that $\sup_{t}\|u^\varepsilon_t\|_{H^1(\mathbb{R})}<\infty$. Thus for any $\alpha>0$, 
\[
	U_\alpha^\varepsilon (\cdot):=\int_0^\infty \mathrm{e}^{-\alpha t}u^\varepsilon_t(\cdot)dt \in H^1(\mathbb{R}).
\]
By performing the Laplace transform  at both sides of \eqref{EQ4RUT}, we obtain 
\[
\EE^\varepsilon_\alpha(U_\alpha^\varepsilon , g)=0, \quad \forall g\in C_c^\infty(\mathbb{R}). 
\]
This indicates $U_\alpha^\varepsilon =0$ in $L^2(\mathbb{R})$. Particularly, for any $\varphi\in C_c^\infty(\mathbb{R})$, 
\[
	\int_0^\infty \mathrm{e}^{-\alpha t} \left(u^\varepsilon_t, \varphi\right)_{L^2(\mathbb{R})}dt=(U^\varepsilon_\alpha, \varphi)_{L^2(\mathbb{R})}= 0,\quad \forall \alpha>0. 
\]  
Since $t\mapsto \left(u^\varepsilon_t, \varphi\right)_{L^2(\mathbb{R})}$ is continuous, we can conclude that $\left(u^\varepsilon_t, \varphi\right)_{L^2(\mathbb{R})}=0$ for any $\varphi\in C_c^\infty(\mathbb{R})$ and $t>0$. Therefore $u^\varepsilon=0$. 
\end{proof}

\subsection{Boundary conditions of the flux at a singular barrier}

Now we consider the convergence of $u^\varepsilon$ as $\varepsilon\downarrow 0$.
The expected limit is the solution to heat equation with the conductivity $a$ in \eqref{EQ4ADA}
\begin{equation}\label{EQ4UTT}
\begin{aligned}
	&\frac{\partial u}{\partial t}(t,x)=\frac{1}{2}\nabla\left(a(x)\nabla u(t,x) \right),\quad t\geq 0,x\in \mathbb{R},\\
	&u(0,\cdot)=u_0. 
\end{aligned}
\end{equation}
Similar to Definition~\ref{DEF411}, the weak solution to \eqref{EQ4UTT} is defined as follows.
\begin{definition}
Given a family $\mathscr{H}$ of space-time functions, $u$ is called a weak solution to \eqref{EQ4UTT}  in $\mathscr{H}$, if $u\in \mathscr{H}$, $u(0,\cdot)=u_0$ and for any $t>0$, $g\in C_c^\infty(\mathbb{R})$, 
\begin{equation}
\int_\mathbb{R}(u_0(x)-u(t,x))g(x)dx=\frac{1}{2}\int_0^t \int_\mathbb{R} a(x)\nabla u(s,x)\nabla g(x)dxds. 
\end{equation}
\end{definition}

Two families of space-time functions for the solutions to \eqref{EQ4UTT} will be considered: 
\[
\begin{aligned}
&\mathscr{H}(\mathbb{R}):= C_b\left([0,\infty),L^2(\mathbb{R})\right)\cap L^\infty\left([0,\infty),H^1(\mathbb{R})\right), \\
&\mathscr{H}(\mathbb{G}):= C_b\left([0,\infty),L^2(\mathbb{G})\right)\cap L^\infty\left([0,\infty),H^1(\mathbb{G})\right).
\end{aligned}
\]
Recall that $H^1(\mathbb{G})=\{u\in L^2(\mathbb{G}): u_\pm\in H^1(\mathbb{G}_\pm)\}$. Every function $u$ in $H^1(\mathbb{G})$ or $C(\mathbb{G})$ may be regarded as a discontinuous function on $\mathbb{R}$, which is continuous on $(-\infty, 0)$ and $(0, \infty)$ respectively and has finite left and right limits at $0$.

\begin{theorem}\label{THM414}
Assume $u_0\in H^1(\mathbb{R})$.  Take a decreasing sequence $\varepsilon_n\downarrow 0$ and write $u^n$ for $u^{\varepsilon_n}$, i.e. the unique weak solution to \eqref{EQ4UVT}. Set 
\[
	\bar{\gamma}(n):=\int_{-\varepsilon_n}^{\varepsilon_n} \frac{1}{b_{\varepsilon_n}(x)}dx.
\] 
Assume $\lim_{n\rightarrow \infty}\varepsilon_n\cdot\bar{\gamma}(n)=0$ and
\[
	\bar{\gamma}:=\lim_{n\rightarrow \infty}\bar{\gamma}(n) \quad  (\leq \infty)
\]
exists.
Then for any $t>0$, the limit $u_t$ of $u^n_t$ exists in $L^2(\mathbb{R})$ as $n\rightarrow \infty$. Furthermore, assume $a\in C(\mathbb{G})$, and write $u(t,x):=u_t(x), U_\alpha (\cdot):=\int_0^\infty \mathrm{e}^{-\alpha t}u_t(\cdot)dt$ for any $\alpha>0$. Then the following assertions hold:
\begin{itemize}
\item[(1)] $\bar{\gamma}=\infty$: $u$ is a weak solution to \eqref{EQ4UTT} in $\mathscr{H}(\mathbb{G})$. For any $\alpha>0$, $U_\alpha $ satisfies the following boundary condition at $0$:
\begin{equation}\label{EQ4UUO}
U'_\alpha (0+)=U'_\alpha (0-)=0. 
\end{equation}
If in addition
\begin{equation}\label{EQ4AXU}
	a(x)u'_0(x)\in H^1(\mathbb{G}),\quad u'_0(0\pm)=0,
\end{equation}
then for any $t>0$, $u_t$ also satisfies the boundary condition at $0$:  
\[
u'_t(0+)=u'_t(0-)=0. 
\]  
\item[(2)] $0<\bar{\gamma}<\infty$: $u$ is a weak solution to \eqref{EQ4UTT} in $\mathscr{H}(\mathbb{G})$. For any $\alpha>0$, $U_\alpha $ satisfies the following boundary condition at $0$ with $\kappa=2/\bar{\gamma}$: 
\begin{equation}\label{EQ4AUA}
a(0+) U'_\alpha (0+)=a(0-) U'_\alpha (0-)=\frac{\kappa}{2}\left(U_\alpha (0+)-U_\alpha (0-)\right). 
\end{equation}
If in addition \eqref{EQ4AXU} holds, then for any $t>0$, $u_t$ also satisfies the boundary condition at $0$:  
\[
a(0+)u'_t(0+)=a(0-)u'_t(0-)=\frac{\kappa}{2}\left(u_t(0+)-u_t(0-)\right).
\]
\item[(3)] $\bar{\gamma}=0$: $u$ is the unique weak solution to \eqref{EQ4UTT} in $\mathscr{H}(\mathbb{R})$. For any $t>0$, $u_t$ is continuous at $0$. 
\end{itemize}
Particularly, the weak solution to \eqref{EQ4UTT} is unique in $\mathscr{H}(\mathbb{R})$ but not unique in $\mathscr{H}(\mathbb{G})$. 
\end{theorem}
\begin{proof}
Note that $u^n_t=P^n_tu_0$, where $P^n_t$ is the semigroup of $X^{\varepsilon_n}$, and the total thermal resistance of $I_{\varepsilon_n}$ with respect to $(\EE^{\varepsilon_n},\FF^{\varepsilon_n})$ is nothing but $\bar{\gamma}(n)$. Then the existence of $u_t$ in $L^2(\mathbb{R})$ follows from Theorem~\ref{THM45} and Proposition~\ref{PRO42}. 

The case $\bar{\gamma}=0$ is clear by mimicking Lemma~\ref{LM412}. Now consider the case $0<\bar{\gamma}<\infty$. Denote 
\[
	d\lambda=\frac{1}{a(x)}dx
\]
and the Dirichlet form \eqref{EQ4FSF} with this $\lambda$ ($m$ is the Lebesgue measure) by $(\EE^\s,\FF^\s)$. Then $u_t=P^\s_tu_0$, where $P^\s_t$ is the semigroup of $X^\s$. Its resolvent is denoted by $R^\s_\alpha$. Mimic the first part of the proof of Lemma~\ref{LM412} for $(\EE^\s,\FF^\s)$ and note that \eqref{EQ4UGL} is replaced by
\[
	(u_0,g)_{L^2(\mathbb{R})}-\alpha(R^\s_\alpha u_0, g)_{L^2(\mathbb{R})}=\EE^\s(R^\s_\alpha u_0,g)=\frac{1}{2}\int_\mathbb{G} a(x)\left(R^\s_\alpha u_0\right)'(x)g'(x)dx, 
\]
because of the continuity of $g$. Then we can conclude that $u$ is a weak solution to \eqref{EQ4UTT} in $\mathscr{H}(\mathbb{G})$. Since $U_\alpha =R^\s_\alpha u_0\in \mathcal{D}(\L^\s)$, where $\L^\s$ is the generator of $(\EE^\s,\FF^\s)$ on $L^2(\mathbb{R})$, it follows from Proposition~\ref{PRO45} that $U_\alpha $ satisfies the boundary condition \eqref{EQ4AUA}. The condition \eqref{EQ4AXU} implies $u_0\in \mathcal{D}(\L^\s)$. By Hille-Yosida theorem, we have $u_t=P^\s_tu_0\in \mathcal{D}(\L^\s)$ and thus $u_t$ also satisfies the same boundary condition at $0$. The proof of the case $\bar{\gamma}=\infty$ is the same as that of $0<\bar{\gamma}<\infty$. The non-uniqueness of weak solutions in $\mathscr{H}(\mathbb{G})$ is clear, since different $\bar{\gamma}$ corresponds to different Markov process. That completes the proof. 
\end{proof}
\begin{remark}
In Corollary~\ref{COR48}, $a\equiv 1$ and $b_\varepsilon(x)=(\kappa \varepsilon)^{-\alpha}$. The three phases in Theorem~\ref{THM414} still correspond to $\alpha<-1$, $\alpha=-1$ and $\alpha>-1$ respectively. Particularly, the boundary condition \eqref{EQ4AUA} for the phase $\alpha=-1$ has been considered in \cite[Proposition~1]{L16}. 
\end{remark}

As shown in Lemma~\ref{LM412}, $u^\varepsilon_t \in H^1(\mathbb{R})$ and this indicates the flux of thermal conduction with a small normal barrier is continuous near $0$. When $\varepsilon\downarrow 0$, the continuity at $0$ still holds unless the total thermal resistance tends to $0$. Otherwise, the flux has a gap between $0-$ and $0+$, and the boundary condition \eqref{EQ4UUO} or \eqref{EQ4AUA} appears. Note that by letting $\bar{\gamma}\uparrow \infty$ in the semi-permeable case, the boundary condition \eqref{EQ4AUA} becomes the impermeable one \eqref{EQ4UUO} formally.

\bibliographystyle{abbrv}
\bibliography{stiff}

\end{document}